\documentclass[reqno, 12pt]{amsart}
\pdfoutput=1

\def\ssign{\textsection\nobreak\hspace{1pt plus 0.3pt}}
\makeatletter
\let\origsection=\section 
\def\mysection{\@mystartsection{section}{1}\z@{.7\linespacing\@plus\linespacing}{.5\linespacing}{\normalfont\scshape\centering\ssign}}
\def\section{\@ifstar{\origsection*}{\mysection}}
\def\appendix{\par\c@section\z@ \c@subsection\z@
   \let\sectionname\appendixname
   \let\section=\origsection 
   \def\thesection{\@Alph\c@section}} 
\def\@mystartsection#1#2#3#4#5#6{\if@noskipsec \leavevmode \fi 
 \par \@tempskipa #4\relax
 \@afterindenttrue
 \ifdim \@tempskipa <\z@ \@tempskipa -\@tempskipa \@afterindentfalse\fi
 \if@nobreak \everypar{}\else
     \addpenalty\@secpenalty\addvspace\@tempskipa\fi
 \@dblarg{\@mysect{#1}{#2}{#3}{#4}{#5}{#6}}}
\def\@mysect#1#2#3#4#5#6[#7]#8{\edef\@toclevel{\ifnum#2=\@m 0\else\number#2\fi}\ifnum #2>\c@secnumdepth \let\@secnumber\@empty
  \else \@xp\let\@xp\@secnumber\csname the#1\endcsname\fi
  \@tempskipa #5\relax
  \ifnum #2>\c@secnumdepth
    \let\@svsec\@empty
  \else
    \refstepcounter{#1}\edef\@secnumpunct{\ifdim\@tempskipa>\z@ \@ifnotempty{#8}{\@nx\enspace}\else
        \@ifempty{#8}{.}{\@nx\enspace}\fi
    }\@ifempty{#8}{\ifnum #2=\tw@ \def\@secnumfont{\bfseries}\fi}{}\protected@edef\@svsec{\ifnum#2<\@m
        \@ifundefined{#1name}{}{\ignorespaces\csname #1name\endcsname\space
        }\fi
      \@seccntformat{#1}}\fi
  \ifdim \@tempskipa>\z@ \begingroup #6\relax
    \@hangfrom{\hskip #3\relax\@svsec}{\interlinepenalty\@M #8\par}\endgroup
    \ifnum#2>\@m \else \@tocwrite{#1}{#8}\fi
  \else
  \def\@svsechd{#6\hskip #3\@svsec
    \@ifnotempty{#8}{\ignorespaces#8\unskip
       \@addpunct.}\ifnum#2>\@m \else \@tocwrite{#1}{#8}\fi
  }\fi
  \global\@nobreaktrue
  \@xsect{#5}}
\makeatother

\usepackage{amsmath,amssymb,amsthm}
\usepackage{mathrsfs}
\usepackage{mathabx}\changenotsign
\usepackage{dsfont}
\usepackage[dvipsnames]{xcolor}

\usepackage{xcolor}
\usepackage{hyperref}
\hypersetup{
    colorlinks,
    linkcolor={red!60!black},
    citecolor={green!60!black},
    urlcolor={blue!60!black}
}

\usepackage[dvipsnames]{xcolor}
\usepackage[open,openlevel=2,atend]{bookmark}

\usepackage[abbrev]{amsrefs}
\usepackage{doi}

\renewcommand{\PrintDOI}[1]{\doi{#1}}

\usepackage[T1]{fontenc}
\usepackage{lmodern}
\usepackage[babel]{microtype}
\usepackage[english]{babel}

\usepackage{geometry}
\numberwithin{equation}{section}
\numberwithin{figure}{section}

\usepackage{enumitem}
\def\rmlabel{\upshape({\itshape \roman*\,})}

\def\alabel{\upshape({\itshape \alph*\,})}

\usepackage{subcaption}
\def\ag#1{	\tikz{\def\nn{#1};
	\pgfmathsetmacro\wie{3*\nn-1};
	\pgfmathsetmacro\mm{5*\nn};
	\pgfmathsetmacro\kk{\nn-1};
	
	\foreach \i in {0,...,\mm}{
		\coordinate (x\i) at (90+\i*360/\wie:1.3cm);}
		
	\foreach \i in {0,...,\wie}{	
		\foreach \j [evaluate=\j as \k using \i+\j+\nn] in {0,...,\kk}{
			\draw[green!75!black] (x\i)--(x\k);}
	}		
	
	\foreach \i in {0,...,\wie}{	
			\draw[black, fill=black] (x\i) circle (1pt);
	}		}}

\newcommand{\mref}[1]{\ifmmode\textrm{\ref{#1}}\else\ref{#1}\fi}

\let\polishlcross=\l
\def\l{\ifmmode\ell\else\polishlcross\fi}

\def\tand{\ \text{and}\ }

\let\vn=\vn
\let\setminus=\smallsetminus
\let\sm=\smallsetminus
\let\wt=\widetilde

\makeatletter
\def\moverlay{\mathpalette\mov@rlay}
\def\mov@rlay#1#2{\leavevmode\vtop{   \baselineskip\z@skip \lineskiplimit-\maxdimen
   \ialign{\hfil$\m@th#1##$\hfil\cr#2\crcr}}}
\newcommand{\charfusion}[3][\mathord]{
    #1{\ifx#1\mathop\vphantom{#2}\fi
        \mathpalette\mov@rlay{#2\cr#3}
      }
    \ifx#1\mathop\expandafter\displaylimits\fi}
\makeatother

\newcommand{\dcup}{\charfusion[\mathbin]{\cup}{\cdot}}
\newcommand{\bigdcup}{\charfusion[\mathop]{\bigcup}{\cdot}}

\DeclareFontFamily{U}  {MnSymbolC}{}
\DeclareSymbolFont{MnSyC}         {U}  {MnSymbolC}{m}{n}
\DeclareFontShape{U}{MnSymbolC}{m}{n}{
    <-6>  MnSymbolC5
   <6-7>  MnSymbolC6
   <7-8>  MnSymbolC7
   <8-9>  MnSymbolC8
   <9-10> MnSymbolC9
  <10-12> MnSymbolC10
  <12->   MnSymbolC12}{}
\DeclareMathSymbol{\powerset}{\mathord}{MnSyC}{180}

\usepackage{tikz}
\usetikzlibrary{calc,decorations.pathmorphing}
\usetikzlibrary{math}
\pgfdeclarelayer{background}
\pgfdeclarelayer{foreground}
\pgfdeclarelayer{front}
\pgfsetlayers{background,main,foreground,front}

\usepackage{subcaption}

\let\epsilon=\varepsilon

\let\rho=\varrho
\let\theta=\vartheta

\def\NN{{\mathds N}}

\def\ZZ{{\mathds Z}}

\newcommand{\cC}{\mathcal{C}}

\newcommand{\cF}{\mathcal{F}}

\newcommand{\cX}{\mathcal{X}}

\newcommand{\cY}{\mathcal{Y}}

\newcommand{\Ex}{\mathfrak{E}	(n,s)} 
\newcommand{\ex}{\mathrm{ex}}
\newcommand{\gF}{\mathfrak{F}	(n,s)}

\newcommand{\Nn}{\mathrm{N}}

\newcommand{\al}{\alpha}

\newcommand{\df}[1]{{\emph{\bf #1}}}

\newtheoremstyle{note}  {4pt}  {4pt}  {\sl}  {}  {\bfseries}  {.}  {.5em}          {}
\newtheoremstyle{introthms}  {3pt}  {3pt}  {\itshape}  {}  {\bfseries}  {.}  {.5em}          {\thmnote{#3}}
\newtheoremstyle{remark}  {2pt}  {2pt}  {\rm}  {}  {\bfseries}  {.}  {.3em}          {}

\theoremstyle{plain}
\newtheorem{thm}{Theorem}[section]
\newtheorem{lemma}[thm]{Lemma}

\newtheorem{conj}[thm]{Conjecture}
\newtheorem{cor}[thm]{Corollary}
\newtheorem{fact}[thm]{Fact}
\newtheorem{claim}[thm]{Claim}

\theoremstyle{note}
\newtheorem{dfn}[thm]{Definition}

\theoremstyle{remark}

\usepackage{accents}

\newcommand{\re}[1]{\textcolor{red!75!black}{#1}}
\newcommand{\bl}[1]{\textcolor{blue!75!black}{#1}}

\def\Fg{\cF_G}
\let\phi=\varphi
\let\vn=\varnothing

\usepackage{mathtools}
\usepackage{lineno}
\newcommand*\patchAmsMathEnvironmentForLineno[1]{\expandafter\let\csname old#1\expandafter\endcsname\csname #1\endcsname
\expandafter\let\csname oldend#1\expandafter\endcsname\csname end#1\endcsname
\renewenvironment{#1}{\linenomath\csname old#1\endcsname}{\csname oldend#1\endcsname\endlinenomath}}\newcommand*\patchBothAmsMathEnvironmentsForLineno[1]{\patchAmsMathEnvironmentForLineno{#1}\patchAmsMathEnvironmentForLineno{#1*}}\AtBeginDocument{\patchBothAmsMathEnvironmentsForLineno{equation}\patchBothAmsMathEnvironmentsForLineno{align}\patchBothAmsMathEnvironmentsForLineno{flalign}\patchBothAmsMathEnvironmentsForLineno{alignat}\patchBothAmsMathEnvironmentsForLineno{gather}\patchBothAmsMathEnvironmentsForLineno{multline}}
\usepackage{multicol}
\usepackage{todonotes}

\newcommand{\expl}[1]\relax

\usepackage{marginnote}

\def\GZS{\mathbf{Sym}}\newcommand{\G}{\Gamma}

\begin{document}

\title[The next case of Andr\'asfai's conjecture]{The next case of 
Andr\'asfai's conjecture}

\author[T.~\L uczak]{Tomasz \L uczak}
\address{Adam Mickiewicz University, Faculty of Mathematics and Computer Science, Pozna\'n, Poland}
\email{tomasz@amu.edu.pl}

\author[J.~Polcyn]{Joanna Polcyn}
\address{Adam Mickiewicz University, Faculty of Mathematics and Computer Science, Pozna\'n, Poland}
\email{joaska@amu.edu.pl}
\thanks{The first author was supported in part by National Science Centre, Poland, grant 2022/47/B/ST1/01517.}

\author[Chr.~Reiher]{Christian Reiher}
\address{Fachbereich Mathematik, Universit\"at Hamburg, Hamburg, Germany}
\email{Christian.Reiher@uni-hamburg.de}


\dedicatory{Dedicated to the memory of Vera T. S\'os}

\begin{abstract}
Let $\ex(n,s)$ denote the maximum number of edges in a triangle-free graph 
on~$n$ vertices which contains no independent sets larger than~$s$.
The behaviour of $\ex(n,s)$ was first studied by Andr\'asfai, 
who conjectured that for $s>n/3$ this function is determined by 
appropriately chosen blow-ups of so called Andr\'asfai graphs. 
Moreover, he proved $\ex(n, s)=n^2-4ns+5s^2$ for $s/n\in [2/5, 1/2]$
and in earlier work we obtained $\ex(n, s)=3n^2-15ns+20s^2$ 
for $s/n\in [3/8, 2/5]$.
Here we make the next step in the quest to settle Andr\'asfai's conjecture 
by proving $\ex(n, s)=6n^2-32ns+44s^2$ for $s/n\in [4/11, 3/8]$. 
\end{abstract}

\maketitle

\section{Introduction}\label{sec:intro}
Ramsey-Tur\'an theory was systematically initiated by Vera~T.~S\'os~\cite{ES69}.
Throughout her long career, she frequently returned to this topic, asked numerous 
intriguing questions, and proved a plethora of deep results
(see e.g.,~\cites{ES82, EHSS, EHSSS93, EHSSS97, MuSo}). 
For a survey on this area we refer to Simonovits and S\'os~\cite{SS}. 

A central function studied in this context, but by no means the most general one, 
is the following: Given integers $\ell\ge 3$ and $n\ge s\ge 0$ 
the {\sl Ramsey-Tur\'an number $\mathrm{ex}_\ell(n,s)$} is the maximum number 
of edges in a $K_\ell$-free graph on $n$ vertices which contains no independent
set consisting of more than~$s$ vertices, i.e.,
\[
	\mathrm{ex}_\ell(n,s)
	=
	\max_{G=(V,E)}\bigl\{|E|\colon K_\ell\nsubseteq G,\;|V|=n, 
	\textrm{ and } \al(G)\le s\bigl\}\,.
\]

In this article we concentrate on triangle-free graphs and, eliminating some 
indices, we study the behaviour of the function $\ex(n,s)=\ex_3(n,s)$.
We also write
\[
	\Ex
	=
	\bigl\{G=(V,E)\colon K_3\nsubseteq G,\;|V|=n, \;\al(G)\le s,
		\textrm{ and } |E|=\ex(n,s)\bigr\}
	\]
	for the corresponding families of extremal graphs. 

Let us start with a few observations. 
Mantel's theorem yields $\ex(n,s)=\lfloor n^2/4\rfloor$ for every $s\ge n/2$.
Since in triangle-free graphs the neighbourhoods of all vertices are 
independent sets, we have the trivial bound
\begin{equation}\label{eq:trivial}
	\ex(n,s)\le sn/2\,,
\end{equation}
which holds with equality if and only if there exists an $s$-regular graph 
on $n$ vertices such that $\al(G)=s$. Brandt~\cite{B10} showed that the set 
\[
	\{s/n\colon \textrm{there is a triangle-free $s$-regular graph $G$ 
		on $n$ vertices with } \alpha(G)=s\}
\]
is dense in the interval $[0,1/3]$; thus, from an asymptotic point of view, 
it suffices to characterise the behaviour of $\ex(n,s)$ for $n/3< s < n/2$. 
The first researcher who studied this problem was Andr\'asfai~\cite{A}; he 
proved that for $2n/5\le s < n/2$ the family $\Ex$ consists of certain 
blow-ups of the pentagon. Here and below by a {\it blow-up} of a graph $G$ 
we mean any graph $\hat G$ obtained from $G$ by replacing each of its 
vertices $v_i$ by a (possibly empty) independent set $V_i$ and joining 
two vertex classes~$V_i$ and~$V_j$ of $\hat G$ by all $|V_i||V_j|$ possible 
edges whenever the pair $v_iv_j$ is an edge of~$G$. Notice that, since we allow some vertex sets in a blow-up to be empty, complete bipartite graphs are 
blow-ups of the pentagon as well.

Andr\'asfai also defined a family of graphs $\{\Gamma_k\}_{k\ge 1}$, which 
are now known as Andr\'asfai graphs. For every $k$, his graph $\Gamma_k$ is the Cayley graph $\bigl(\ZZ/(3k-1)\ZZ, \{k,k+1,\dots, 2k-1\}\bigr)$, which means that $V(\G_k)=\ZZ/(3k-1)\ZZ$ and 
\[
	E(\G_k)
	=
	\bigl\{ij\in V(\G_k)^{(2)}\colon i-j\in \{k, k+1, \dots, 2k-1\}\bigr\}\,.
\]
Clearly, $\Gamma_k$ is a $k$-regular graph on $3k-1$ vertices.
Since $\{k,k+1,\dots, 2k-1\}$ is a sum-free subset of the 
group $\ZZ/(3k-1)\ZZ$, Andr\'asfai's graphs are also triangle-free. 
The first few graphs in the sequence are $\Gamma_1=K_2$, the pentagon $\G_2$, 
and the Wagner graph $\G_3$ (see also Figure~\ref{fig000}).

\begin{figure}[ht]
		\centering
		\begin{multicols}{5}
			\ag{2} \\ \ag{3} \\ \ag{4} \\ \ag{5} \\\ag{6}
		\end{multicols}
		\caption{Andr\'asfai graphs  $\G_2$, $\G_3$,  $\G_4$, $\G_5$, and $\G_6$.}
		\label{fig000}
\end{figure}

It is well-known and not hard to verify that the independence number of $\G_k$ 
is equal to its degree. Balanced blow-ups of $\G_k$ have the same property and 
thus
\[
	\ex\big ((3k-1)m, km\big)=k(3k-1)m^2/2\,.
\]
Consequently, if $s=kn/(3k-1)$ for some $k\ge 1$, then $\ex(n,s)$ is determined by the blow-up of an Andr\'asfai graph. Andr\'asfai conjectured that this holds,
much more generally, whenever $s>n/3$ (we refer to~\cite{LPR} for more 
information on this conjecture).

Whenever $kn/(3k-1)\le  s< (k-1)n/(3k-4)$ we define the ``canonical'' blow-up 
$G(n;k,s)$ of $\Gamma_k$ by replacing the vertices $1$, $k$, $2k$ by sets 
of $(k-1)n-(3k-4)s$ vertices each, and the remaining $3k-4$ vertices 
of $\G_k$ by sets of $3s-n$ vertices each. (The case $k=4$ is drawn in
Figure~\ref{fig001}.)

\begin{figure}[ht]
\begin{tikzpicture}[scale=.9]
\foreach \i in {0,..., 20} \coordinate (x\i) at (122.73-32.73*\i:3);

\foreach \i in {0,...,10}{
	\foreach \j [evaluate=\j as \k using \i+\j] in {4,...,7}{
			\draw[black!20, line width=6] (x\i)--(x\k);}
}		
\foreach \i in {0,...,10}{
	\foreach \j [evaluate=\j as \k using \i+\j] in {4,...,7}{
		\draw[thick] (x\i)--(x\k);}
}		
\foreach \i in {2,3,5,6,7,9,10,11}{
			\draw[blue!75!black, thick, fill=blue!15] (x\i) ellipse (.8cm and 14pt);
			\node at (x\i) {\large\textcolor{blue!75!black}{$V_{\i}$}};
}	
\foreach \i in {1,4,8}{
			\draw[red!75!black, thick, fill=red!20] (x\i) ellipse (.5cm and 9pt);
			\node at (x\i) {\textcolor{red!75!black}{$V_{\i}$}};
}	
		
		\end{tikzpicture}
		\caption{The canonical blow-up $G(n;4,s)$. Each `small' set 
		$V_1$, $V_4$, $V_8$ contains $3n-8s$ vertices, the remaining 
		`large' sets consist of $3s-n$ vertices.}
		\label{fig001}
\end{figure}

Elementary calculations show that $G(n;k,s)$ has $n$ vertices, 
independence number $s$, and 
\begin{equation}\label{eq:1727}
		g_k(n, s)
		=
		\tfrac12 k(k-1)n^2-k(3k-4)ns+\tfrac12(3k-4)(3k-1)s^2
\end{equation}
edges (see~\cite{LPR}*{Fact~1.5}). 
It can be shown (cf.\ \cite{LPR}*{Fact~2.6}) that 
\begin{equation}\label{eq:fact26}
	\ex(n,s) \le g_k(n,s)\quad\textrm{ whenever\ } s\notin \bigl(\tfrac{k}{3k-1}n, \tfrac{k-1}{3k-4}n\bigr)\,.
\end{equation}
Furthermore, if $s\in \bigl(\tfrac{k}{3k-1}n, \tfrac{k-1}{3k-4}n\bigr)$, then 
among all blow-ups of $\G_k$ with~$n$ vertices and independence number~$s$ 
the canonical blow-up~$G(n;k,s)$ has the maximum number of 
edges (cf.\ \cite{Vega}*{Lemma~3.3}). Thus, 
Andr\'asfai's conjecture admits the following formulation.

\begin{conj}\label{c:1}
	For all integers $k$, $s$, and $n$ such that 
	$kn/(3k-1)\le  s\le(k-1)n/(3k-4)$ we have 
		\[
		\ex(n,s)=g_k(n,s)\,.
	\]
	\end{conj}

Andr\'asfai himself~\cite{A} proved this for $k=2$, in~\cite{LPR} we added
the case $k=3$, and in~\cite{Vega} we showed that for every $k$ there 
is a small constant $\gamma_k>0$ such that the conjecture holds whenever 	
$kn/(3k-1)\le  s<(k/(3k-1)+\gamma_k)n$. Here we make a further step towards 
the solution of Andr\'asfai's conjecture by addressing the case $k=4$. 
 
\begin{thm}\label{th:main} 
	If $4n/11\le s\le 3n/8$, then 
	$\ex(n,s)= g_4(n,s)=6n^2-32ns+44s^2$.
\end{thm}

We would like to point out that, in general, $G(n;k,s)$ is not the only 
graph in the family~$\Ex$. For instance, in Figure~\ref{fig001} one can 
move some vertices between the two sets~$V_4$ and~$V_5$. Provided that 
the cardinalities of the resulting sets~$V'_4$ and~$V'_5$ remain in the 
interval $[3n-8s, 3s-n]$ the graph keeps having the independence number $s$ 
and the number of edges remains constant as well. For more information 
on the families $\Ex$ we refer to~\cite{Vega}.

We use standard graph theoretical notation. 
For a graph $G$ we denote by~$V(G)$ and~$E(G)$ its vertex and edge set,
respectively, and we write $\nu(G)=|V(G)|$, $e(G)=|E(G)|$.
 If $A, B\subseteq V(G)$ are disjoint,
we write $e_G(A, B)$ for the number of edges in $G$ from $A$ to~$B$ 
and~$e_G(A)$ refers to the number of edges induced by $A$.  
By~$\Nn_G(B)$ we mean the neighbourhood of a set $B\subseteq  V(G)$.
For every $v\in V(G)$ we abbreviate $\Nn_G(\{v\})$ to $\Nn_G(v)$ 
and we put $\deg_G(v)=|\Nn_G(v)|$.
The size of a largest independent set in $G$ is denoted by~$\alpha(G)$. 
Moreover, we often omit subscripts unless they are necessary to avoid confusion. 

If $A$ and $B$ are two disjoint sets, then $K(A, B)$ denotes the edge set 
of the complete bipartite graph with vertex partition $A\dcup B$, i.e., 
\[
	K(A, B) = \bigl\{ab\colon a\in A \text{ and } b\in B\bigr\}\,.
\]
Given a set $X$ we write $\powerset(X)$ for its power set and~$X^{(2)}$
for the collection of all two-element subsets of~$X$. 

\section{Moulds}

In this section we introduce the rough idea behind the proof of Theorem~\ref{th:main}. 
Let us start with a few definitions. First, for any integers $n\ge s\ge 0$ 
it will be convenient to write 
\[
	\gF=\bigl\{G=(V,E)\colon K_3\nsubseteq G,\;|V|=n, 
		\textrm{ and }\al(G)\le s\bigr\}
\]
for the class of triangle-free graphs on $n$ vertices whose independence number is 
at most~$s$. The extremal graphs $\Ex$ form a subclass of $\gF$, and we shall illustrate some of our concepts by looking at the canonical blow-up $G(n; 4, s)$, 
which belongs to $\gF$ as well.  

\begin{dfn}	The \df{fortress} $\Fg$ of a graph $G\in \gF$ is the graph with 
		\[
		V(\Fg) 
		= 
		\{X\subseteq V(G)\colon |X|=s \tand X \textrm{ is independent}\}
	\]
		and 
		\[
	E(\Fg) = \{XY\in V(\Fg)^{(2)}\colon X\cap Y=\vn\}\,.
	\]
		If $G$ is clear from the context, we will often write $\cF$ instead 
	of $\Fg$.		
\end{dfn}

For instance, if $s/n\in (4/11, 3/8)$, then the fortress of the canonical graph
$G(n;4,s)$ (see Figure~\ref{fig001}) has $10$ vertices, 
namely the neighbourhoods of the vertex classes $V_i$ with $i\ne 8$. 
More precisely, the neighbourhoods of the large vertex classes form 
a copy of $\Gamma_3$ in $\cF$ and $\Nn(V_1)$, $\Nn(V_4)$ are
twins of $\Nn(V_{11})$, $\Nn(V_5)$, respectively. The fortresses of other 
extremal graphs obtained by moving some vertices from $V_5$ to $V_4$
have only nine vertices (because~$\Nn(V_1)$ disappears), but these fortresses
still contain copies of $\Gamma_3$.  

\begin{dfn}\label{df:prehypo}
	Let $G\in\gF$ and $H$ be two graphs. An \df{imprint} of $H$ in $G$ is 
	an embedding of $H$ into $\Fg$, i.e., an injective map  
	\[
	\varphi \colon V(H) \longrightarrow V(\Fg)
	\]	
	such that 
	\begin{enumerate}[label=\textit{(M\arabic*)}]
		\item \label{it:h2} $\forall x,y \in V(H) \,\, [xy\in E(H) \,\,\,
		\Longleftrightarrow \,\,\, \varphi(x)\varphi(y)\in E(\Fg)].$ 
	\end{enumerate}	
\end{dfn}

The canonical blow-ups $G(n;4,s)$ have much more structure than just an imprint 
of $\G_3$ --- in addition to that, the eight large vertex classes form a blow-up
of $\G_3$. This interplay of an imprint with a blow-up is captured by the notion 
of a mould. 
  
\begin{dfn}\label{df:hypos}
	Given two graphs $G\in\gF$ and $H$ an \df{$H$-mould} for $G$ is
   a pair $(\varphi,\psi)$ of maps from $V(H)$ to $\powerset(V(G))$ 
   such that $\varphi$ is an imprint and for all $x\in V(H)$,
	\begin{enumerate}[label=\textit{(M\arabic*)}]\setcounter{enumi}{1}
		\item\label{it:h3} $\psi(x)$ is an independent set of size $3s-n$;
		\item\label{it:h4} $K(\varphi(x), \psi(x))\subseteq E(G)$.
	\end{enumerate}
	If some such $H$-mould exists for $G$, we say that $G$ contains an $H$-mould. 
\end{dfn}

When discussing moulds, we always assume tacitly that $s>n/3$, which 
causes the sets~$\psi(x)$ to be nonempty. 
Here are some immediate consequences of the definition of a mould. 

\begin{fact}\label{f:24}
	If $(\varphi,\psi)$ is an $H$-mould for some $G\in \gF$ and $x\in V(H)$,  
	then every $z\in \psi(x)$ satisfies $\Nn_G(z)=\varphi(x)$.
\end{fact}

\begin{proof}
	Due to~\ref{it:h4} we have $\Nn_G(z)\supseteq\varphi(x)$ and because of 
	$|\Nn_G(z)|\le \alpha(G)\le |\varphi(z)|$ this needs to hold with equality.
\end{proof}

\begin{fact}\label{f:00}
	Let $(\varphi,\psi)$ be an $H$-mould for some $G\in \gF$. 
	If $x,y\in V(H)$ are distinct, then $\psi(x)\cap \psi(y)=\vn$. 
\end{fact}

\begin{proof}
	Let us suppose that $v\in \psi(x)\cap \psi(y)$. Now Fact~\ref{f:24} tells us
	$\varphi(x)=\Nn(v)=\varphi(y)$, which contradicts $x\neq y$ (recall 
	that $\phi$ is required to be injective). 
\end{proof}

\begin{fact}\label{f:45}
	Let $(\varphi,\psi)$ be an $H$-mould for some $G\in \gF$. If $X\in V(\Fg)$ 
	denotes an independent set of size $s$ and $y\in V(H)$, then the following 
	statements are equivalent.
	\begin{enumerate}[label=\rmlabel]
	\item \label{it:45i}$\psi(y)\subseteq X$,
	\item\label{it:45ii}$\psi(y)\cap X \neq\vn$,
	\item\label{it:45iii} $\varphi(y)\cap X =\vn$.
		\end{enumerate}
\end{fact}
\begin{proof} 
	Clearly, \ref{it:45i} implies \ref{it:45ii}. 
	Next, if $v\in \psi(y)\cap X$, then $\Nn(v)=\phi(y)$ needs to be 
	disjoint to the independent set $X$ and so~\ref{it:45ii} 
	implies~\ref{it:45iii}. 

	Finally, assume that $\varphi(y)\cap X=\vn$. 
	Since $\Nn(\psi(y))=\varphi(y)$, the set $\psi(y)\cup X$ is independent. 
	In view of $|X|=s=\alpha(G)$ this shows $\psi(y)\subseteq X$. 
\end{proof}

\begin{fact}\label{f:use}
	If $(\varphi,\psi)$ is an $H$-mould for some $G\in \gF$, then the 
	subgraph of $G$ induced by $\bigcup_{v\in V(H)}\psi(v)$ is the blow-up
	of $H$ obtained by replacing every vertex $v\in V(H)$ by the vertex 
	class $\psi(v)$ of size $3s-n$.  
\end{fact}

\begin{proof}
	Simplifying the notation we suppose $V(H)=\{1,\dots, \ell\}$ 
	for some natural number~$\ell$ and we set $A_i=\varphi(i)$, $B_i=\psi(i)$ 
	for every $i\in [\ell]$. 
	By Fact~\ref{f:00} the sets $B_1, \dots, B_\ell$ are mutually disjoint.  
	Thus it is enough to show the following statements for every 
	pair~$ij\in [\ell]^{(2)}$. 
		\begin{enumerate}[label=\alabel]
		\item\label{it:26a} If $ij \in E(H)$, then $K(B_i, B_j)\subseteq E(G)$.
		\item\label{it:26b} If there exists an edge $b_ib_j\in E(G)$ 
			with $b_i\in B_i$ and $b_j\in B_j$, then $ij\in E(H)$.
	\end{enumerate}
		Suppose first that $ij\in E(H)$, which means $A_i\cap A_j=\vn$.
 	By Fact~\ref{f:45} applied to $B_i, A_i, A_j$ here in place 
	of $\psi(y), \varphi(y), X$ there this yields $B_i\subseteq A_j$,
	whence 
		\[
		K(B_i,B_j)\subseteq K(A_j, B_j)\subseteq E(G)\,.
	\]
		
	Having thus proved~\ref{it:26a} we proceed with~\ref{it:26b}.
	Since $G$ is triangle-free, 
	$b_ib_j \in E(G)$ entails $\Nn(b_i)\cap \Nn(b_j)=\vn$. 
	So $A_i$ and $A_j$ are disjoint and, therefore, $ij\in E(H)$.
\end{proof}

Recall that the canonical blow-ups $G(n;4,s)$ contain $\G_3$-moulds. 
The starting point of our proof of Theorem~\ref{th:main} is the realisation 
that, in fact, such a mould is all one needs for counting edges.
  
\begin{lemma}\label{l:00}
	If $4n/11< s<3n/8$ and some graph $G\in \Ex$ contains a $\G_3$-mould, 
	then 	$\ex(n,s)\le g_4(n,s)$, i.e., the pair $(n, s)$ satisfies 
	Theorem~\ref{th:main}. 
\end{lemma}

\begin{proof}
	Let $(\varphi,\psi)$ be a $\G_3$-mould for $G\in \Ex$, 
	and set $B_i=\psi(i)$ for $i=1,2,\dots,8$. 
	Our aim is to show that $G$ has at most as many edges as the canonical 
	blow-up $H=G(n;4,s)$. The calculations given below become almost obvious 
	if one looks at Figure~\ref{fig001}. 
	
	Note that the eight sets $B_i$ are pairwise disjoint, consist of $3s-n$ 
	vertices each, and by Fact~\ref{f:use} their union 
	$W=\bigdcup_{i\in [8]} B_i$
	induces a balanced blow-up of $\G_3$ in~$G$. We shall compare~$W$ with 
	the union $L$ of the eight large vertex classes of $H$ and its complement 
	$\bar W=V(G)\setminus W$ with the union of the three small 
	sets, i.e., with $\bar L=V(H)\sm L$. Due to $|W|=|L|=8(3s-n)$ 
	we have $|\bar W|=|\bar L|=3t$, where $t=3n-8s$ denotes the common size 
	of the small sets of $H$.
	
	It is plain that $e_G(W)=12(3s-n)^2=e_H(L)$.
	Next, every vertex $w\in W$ has degree $s$ and $3(3s-n)$ 
	neighbours of $w$ are themselves in $W$. 
	Thus $w$ sends $s-3(3s-n)=t$
	edges to $\bar W$, which shows that the number $e_G(W,\bar W)$ of edges 
	in $E(G)$ joining $W$ to $\bar W$ is 
		\begin{equation}\label{eq:1952}
		e_G(W,\bar W)=8(3s-n)t=e_H(L,\bar L)\,.
	\end{equation}
		Consequently, it remains to prove 
		\begin{equation}\label{eq:1955}
		e_G(\bar W)\le 2t^2=e_H(\bar L)\,.
	\end{equation}
		
	To this end we consider the maximum degree $\Delta$ of $G[\bar W]$.
	Since $G$ is triangle-free, a well-known strengthening of Mantel's 
	theorem yields $e_G(\bar W)\le \Delta(3t-\Delta)=2t^2-(\Delta-2t)(\Delta-t)$.
	So if $\Delta\ge 2t$ we are done immediately and henceforth we can suppose 
		\begin{equation}\label{eq:1957}
		\Delta\le 2t\,.
	\end{equation}
			
	Now we revisit~\eqref{eq:1952} from the perspective of $\bar W$. 
	For every vertex $\bar w\in\bar W$ the set $\Nn(\bar w)\cap W$ is a 
	(possibly empty) union of several of the sets $B_i$. In fact, it can 
	be the union of at most $\alpha(\G_3)=3$ such sets. Let us call $\bar w$ 
	\df{heavy} if $\Nn(\bar w)\cap W$ consists of exactly three sets $B_i$, 
	so that $|\Nn(\bar w)\cap W|=3(3s-n)$. Clearly, if $\bar w\in\bar W$
	fails to be heavy, then $|\Nn(\bar w)\cap W|\le 2(3s-n)$. Concerning the 
	number $a$ of heavy vertices we thus obtain
		\[
		e_G(W, \bar W)
		=
		\sum_{\bar w\in \bar W}|\Nn(\bar w)\cap W|
		\le 
		(3s-n)[3a+2(3t-a)]
		=
		(3s-n)(6t+a)\,,
	\]
		which together with~\eqref{eq:1952} shows $8t\le 6t+a$, i.e., $a\ge 2t$.
	
	As each heavy vertex has at most $s-3(3s-n)=t$ neighbours in $\bar W$,
	this yields  
		\[
		2e_G(\bar W)
		=
		\sum_{\bar w\in \bar W}|\Nn(\bar w)\cap \bar W|
		\overset{\eqref{eq:1957}}{\le} 
		at+(3t-a)\cdot 2t
		=
		6t^2-at\le 4t^2\,,
	\]
		which proves~\eqref{eq:1955}.
\end{proof}

Observe that not all graphs $G\in \Ex$ have $\G_3$ as their moulds. 
Indeed, if starting from $G(n;4,s)$ we move some vertices between $V_5$ 
and $V_4$ such that both new sets have fewer than $3s-n$ vertices, 
then the resulting graph only contains an imprint of $\G_3$, but not 
a $\G_3$-mould. 

This state of affairs motivates us to establish that if there is an imprint  
of $\G_k$ in some graph belonging to $\Ex$, then there also exists a graph 
in $\Ex$ containing a $\G_k$-mould (see Lemma~\ref{lem:43} below). 
The proof hinges on a symmetrisation procedure the basic idea of which
can be traced back to the work of Zykov~\cite{Zy} on Tur\'an's theorem. 
Here we follow the notation introduced in~\cite{LPR}. 
 
Given a graph~$G$ and two disjoint sets $A,B\subseteq V(G)$, we say that 
a graph~$G'$ on the same vertex set as~$G$ arises from~$G$ by the 
\df{generalised Zykov symmetrisation $\GZS(A, B)$} if it is obtained by 
deleting all edges incident with $B$ and afterwards adding all edges 
from~$A$ to~$B$. Explicitly, this means
\[
	V(G')=V(G) 
	\quad \text{ and } \quad 
	E(G') 
	= 
	\bigl(E(G)\setminus \{e\in E(G)\colon e\cap B\neq\vn\}\bigr)\cup K(A,B)\,.
\]
We shall write $G'=\GZS(G \,|\, A, B)$ in this situation.
Let us now state a straightforward consequence of \cite{LPR}*{Fact 2.4}
and \cite{LPR}*{Lemma 2.5}. 

\begin{lemma}\label{lem:B''}
	Given two integers $n\ge 0$ and $s\in [n/3, n/2]$, let $G\in\Ex$, and 
	let $A_1,A_2\subseteq V(G)$ be any two disjoint independent sets of size $s$.
	If $M$ is a matching from $V(G)\setminus (A_1\cup A_2)$ 
	to $A_2$ whose size is as large as possible 
	and $B'\subseteq A_2\setminus V(M)$, 
	then $G'=\GZS(G\,|\, A_1, B')$ is again in $\Ex$. \qed
\end{lemma}

In practice there might be several possible choices for the maximum matching $M$
and thus one can try to exercise some control over the location of the set $B'$.
There will be one step in a later argument where we actually need this freedom, 
but in all other cases only the cardinality of $B'$ will matter. Due to 
$|V(G)\setminus (A_1\cup A_2)|=n-2s$ we will automatically 
have $|A_2\setminus V(M)|\ge s-(n-2s)=3s-n$ and thus we can always achieve
$|B'|=3s-n$. It seems worthwhile to state this separately.  

\begin{cor}\label{lem:B'}
	If $s\in [n/3, n/2]$, $G\in\Ex$ and $A_1,A_2\subseteq V(G)$ are two 
	disjoint independent sets of size $s$, then there exists a set 
	$B\subseteq A_2$ of size $|B|=3s-n$ such that $\GZS(G|A_1,B)\in \Ex$. \qed
\end{cor}

To mention just one example of this procedure, consider a graph $G\in \Ex$ 
obtained from $G(n;4,s)\in \Ex$ by moving some vertices from $V_5$ to $V_4$. 
For $A_1= V_9\cup V_{10}\cup V_{11}\cup V_1$ 
and $A_2=V_4\cup V_5\cup V_6\cup V_{7}$ there has to exist a 
set $B\subseteq A_2$ of size $3s-n$ such that $G'=\GZS(G|A_1,B)$ is again an 
extremal graph. Since $\Nn_G(V_9)\cup B$ and $\Nn_G(V_1)\cup B$ are independent 
in~$G'$, and hence of size at most $s$, we need to have $V_5\subseteq B\subseteq V_4\cup V_5$. 
Consequently,~$G'$ 
is isomorphic to the canonical blow-up $G(n;4,s)$. 
Thus we can use symmetrisation in order to `canonise' graphs from $\Ex$.
We apply this technique to show the following result.

\begin{lemma}\label{lem:43}
	Let integers $n\ge 0$, $s\in (n/3, n/2]$, and $k\ge 1$ be given. 
	If some graph in~$\Ex$ contains an imprint of $\G_k$, then $\G_k$ 
	is a mould for some $G\in \Ex$. 
\end{lemma}

\begin{proof}
	Let us recall that $V(\G_k)=\ZZ/ (3k-1)\ZZ$ and 
				\[
			E(\G_k) 
			= 
			\bigl\{ij\in V(\G_k)^{(2)}\colon 
			i-j\in \{k, k+1, \dots, 2k-1\}\bigr\}\,.
		\]
			We call a number $r\le 3k-1$ \df{nice} if there is a graph $G\in \Ex$ 
	containing independent sets $A_1, \dots, A_{3k-1}$ of size $s$ 
	and independent sets $B_1, \dots, B_r$ of size $3s-n$ such that 
	\begin{enumerate}[label=\rmlabel]
		\item\label{it:i} $ij\in E(\G_k) \Longleftrightarrow A_i\cap A_j = \vn$;
		\item\label{it:ii} $B_\ell \subseteq A_{\ell+k}\cap A_{\ell+k+1}\cap \dots\cap A_{\ell+2k-1}$ for all $\ell \in [r]$;
		\item\label{it:iii}  $K(A_\ell, B_\ell)\subseteq E(G)$ for all $\ell \in [r]$.
	\end{enumerate}
	
	For $r=0$ the conditions~\ref{it:ii} and~\ref{it:iii} are void and~\ref{it:i} 
	states that $G$ contains an imprint of $\G_k$. So, by our assumption, 
	$0$ is nice. Let $r_\star$ be the largest nice number. Our goal 
	is to  show that  $r_\star=3k-1$, since then the maps $\varphi(i)= A_i$ 
	and $\psi(i)= B_i$ will certify that $\G_k$ is a mould for $G$, 
	and we are done. 

	Thus, let us assume that $r_\star<3k-1$. As the sets $A_{r_\star+1}$ 
	and $A_{r_\star+k+1}$ are disjoint, Corollary~\ref{lem:B'} delivers a set 
	$B_{r_\star+1} \subseteq A_{r_\star+k+1}$ of size $3s-n$ 
	such that $\GZS(A_{r_\star+1}, B_{r_\star+1})$ produces a 
	graph $G^\circ$ in $\Ex$. Since $A_{r_\star+1}$ is disjoint 
	to $A_{r_\star+k+1},\dots, A_{r_\star+2k}$, these sets stay independent in $G^\circ$
	and Fact~\ref{f:45} applied to $G^\circ$ yields
		\begin{enumerate}[label=\alabel]
		\item\label{it:a} $B_{r_\star+1}\subseteq A_{r_\star+k+1}\cap \dots\cap A_{r_\star+2k}$. 
	\end{enumerate}
		Each of the vertices $r_\star+ 2k+1, \dots, r_\star+k$ of $\G_k$ is adjacent 
	to one of $r_\star+k+1, \dots, r_\star+2k$ and thus~\ref{it:i} 
	and~\ref{it:a} yield
		\begin{enumerate}[label=\alabel, resume]
		\item\label{it:b} $B_{r_\star+1}$ is disjoint 
			to $A_{r_\star+2k+1}, \dots, A_{r_\star+k}$.
	\end{enumerate}
		Altogether, every $A_i$ is disjoint to 
	either $A_{r_\star+1}$ or $B_{r_\star+1}$, whence
		\begin{enumerate}[label=\alabel, resume]
		\item\label{it:c} $A_1, \dots, A_{3k-1}$ are independent in $G^\circ$.
	\end{enumerate}
		
	Next, for every $\ell\in [r_\star]$ there is some 
	$j\in \Nn_{\G_k}(\ell)\sm \Nn_{\G_k}(r_\star+1)$.
	Since~\ref{it:ii} and~\ref{it:b} yield $B_\ell\subseteq A_j$ as well 
	as $A_j\cap B_{r_\star+1}=\vn$, this proves that 
		\begin{enumerate}[label=\alabel, resume]
		\item\label{it:d} $B_{r_\star+1}$ is disjoint to $B_1, \dots, B_{r_\star}$.
	\end{enumerate}
		
	We contend that
	\begin{enumerate}[label=\alabel, resume]
		\item\label{it:e} $K(A_\ell, B_\ell)\subseteq E(G^\circ)$ for 
			every $\ell\in [r_\star]$.
	\end{enumerate}
		If $\{\ell, r_\star+1\}\not\in E(\G_k)$ this is because~\ref{it:b} 
	and~\ref{it:d} entail $(A_\ell\cup B_\ell)\cap B_{r_\star+1}=\vn$, 
	which means that the edges provided by~\ref{it:iii} are not affected 
	by the symmetrisation at all.
	On the other hand, if $\{\ell, r_\star+1\}\in E(\G_k)$, then~\ref{it:ii} 
	and~\ref{it:a} imply $B_\ell\subseteq A_{r_\star+1}$ and 
	$B_{r_\star+1}\subseteq A_\ell$. So the symmetrisation first removes the 
	edges in $K(B_\ell, B_{r_\star+1})$ and then they are immediately put back, 
	so that altogether~\ref{it:iii} remains valid. Thereby~\ref{it:e} is proved.

	Clearly $K(A_{r_\star+1}, B_{r_\star+1})\subseteq E(G^\circ)$ and 
	together with~\ref{it:a},~\ref{it:c},~\ref{it:e} this shows 
	that~$G^\circ$ exemplifies that~$r_\star+1$ is nice, 
	contrary to the maximality of~$r_\star$.
\end{proof}

So the hypothesis of Lemma~\ref{l:00} can be weakened considerably. 

\begin{cor}\label{cor:210}
	If $4n/11< s<3n/8$ and some graph $G\in \Ex$ contains a $\G_3$-imprint, 
	then 	$\ex(n,s)\le g_4(n,s)$. \qed
\end{cor}

\section{Basic fortress tricks}

In this section we take a closer look at the fortresses of extremal graphs. 
The observation that a triangle in a fortress is tantamount to three mutually 
disjoint independent sets of size $s$ has the following consequence. 

\begin{fact}\label{fact:c3}
	Whenever $s>n/3$ and $G\in\Ex$ the fortress of $G$ is triangle-free. \qed
\end{fact}

For Corollary~\ref{cor:210} to be useful we need to argue that, at least 
under some inductive assumptions on $n$ and $s$, one can build a $\G_3$-imprint 
in some graph $G\in\Ex$. The very first step in this direction was undertaken 
in~\cite{LPR}*{Lemma~2.8}, where we proved such a statement 
for $\Gamma_1=K_2$ rather than $\G_3$. 

\begin{lemma}\label{lem:twosets}
	If $n$ and $s$ are two integers with $n\ge 2s \ge 0$, then every 
	graph $G\in \Ex$ contains two disjoint independent sets of size $s$. \qed 
\end{lemma}

If, starting from an edge, one attempts to obtain larger subgraphs
of a fortress $\cF$, it would be helpful to know some construction principles
telling us that certain configurations in $\cF$ are extendible in a 
somewhat controlled manner. One of the most powerful of these principles 
we are aware of (see Corollary~\ref{cor:n} below) can be shown to apply 
to a given pair $(n, s)$ if Theorem~\ref{th:main} itself is assumed to be valid 
for $(n+3, s+1)$. This rules out a straightforward induction on $n$ or $s$, 
but we can still argue by induction on $11s-4n$, which is what we shall 
actually do.

\begin{dfn}
	Set $\delta(n, s)=11s-4n$ for all nonnegative integers $n$ and $s$. 
	The pair $(n, s)$ is said to be a \df{minimal counterexample}
	if $\ex(n, s)>g_4(n, s)$ and 
		\[
		\ex(n', s')\le g_4(n', s')
		\quad 
		\text{ for all $(n', s')$ with } \delta(n', s')<\delta(n, s)\,.
	\]
	\end{dfn}

Let us recall at this juncture that if the estimate $\ex(n, s)\le g_4(n, s)$
fails for any pair of integers $n\ge s\ge 0$, then~\eqref{eq:fact26} yields
\begin{equation}\label{eq:4711}
	\frac{4n}{11}<s<\frac{3n}8\,,
\end{equation}
whence $\delta(n, s)>0$. This means that if minimal counterexamples do not 
exist, then $\ex(n, s)\le g_4(n, s)$ holds unconditionally and together with 
the existence of the canonical blow-ups $G(n; 4, s)$ mentioned in the 
introduction this completes the proof of Theorem~\ref{th:main}.

Thus it suffices to show in the remainder of this article that there are 
no minimal counterexamples. When studying such a minimal counterexample $(n, s)$
we will write $\delta$ instead of $\delta(n, s)$ and the 
bounds~\eqref{eq:4711} will often be used without further referencing.  

\begin{lemma}\label{cl:1}
	Let $(n, s)$ be a minimal counterexample, $G\in \Ex$, 
	and $XY\in E(\cF)$. If $C$ denotes a further independent set in $G$ 
	with 	
		\[
			|C| > 4n - 10s\,,
	\]
		then there exists some $Z\in  V(\cF)$ such that $C\cap Z = \vn$ 
	and either $XZ\in E(\cF)$ or $YZ\in E(\cF)$.
\end{lemma}

\begin{proof}
	Construct a graph $G^+$ by adding three new vertices $x, y, c$ to $G$ 
	as well as all edges from $x$ to $X$, from $y$ to $Y$, from $c$ to $C$, 
	and finally the edge $xy$ (see Figure~\ref{fig:G+3}). 
	Clearly $K_3\nsubseteq G^+$ and $\nu(G^+) = n+3$.
	
		\begin{figure}[ht]
	\centering
	\begin{tikzpicture}[scale=.8]

			\coordinate (b1) at (-2.2,0);
			\coordinate (b2) at (0,0);
			\coordinate (b3) at (2,0);
			
			\coordinate (a1) at (-2.2,1.3);
			\coordinate (a2) at (0,1.3);
			\coordinate (a3) at (2,1.3);

		\fill [red!15, rotate=0] (a1)--($(b1)+(.97,.16)$) -- ($(b1)-(.97,-.16)$) -- cycle;
		\fill [blue!10, rotate=0 ] (a2)--($(b2)+(.97,.16)$) -- ($(b2)-(.97,-.16)$) -- cycle;
		\fill [green!15] (a3)--($(b3)+(.6,.05)$) -- ($(b3)-(.6,.05)$) -- cycle;

			\fill (a1) circle (2pt);		
			\draw [red!70!black,rotate=0, thick] (b1) ellipse (1cm and .5cm);	
			\fill [red!10, rotate=0] (b1) ellipse (1cm and .5cm);	
		
		\fill (a2) circle (2pt);		
	\draw [blue!70!black,rotate=0, thick] (b2) ellipse (1cm and .5cm);	
	\fill [blue!10, rotate=0] (b2) ellipse (1cm and .5cm);

		\fill (a3) circle (2pt);		
		\draw [green!70!black,thick] (b3) ellipse (.6cm and .3cm);	
		\fill [green!10!white] (b3) ellipse (.6cm and .3cm);

		\draw  (a1)--(a2);

	\node at (b2) {$Y$};
	\node at (b1) {$X$};
	\node at (b3) {$C$};

		\node [right] at (a2) {\tiny $y$};
		\node [right] at (a3) {\tiny $c$};
		\node [left] at (a1) {\tiny $x$};

			\node at (-4,.8) {\large\textcolor{black}{$G^+$}};
		
	\end{tikzpicture}
	\caption{The graph $G^+$. In reality, $C$ can intersect $X$ or $Y$.}
	\label{fig:G+3}
\end{figure}
 		
	Due to $\delta(n+3, s+1)=\delta-1$ the minimality of $(n, s)$ yields 
		\begin{align*}
		\ex(n+3, s+1) &\le g_4(n+3, s+1) = g_4(n,s) + 2s + (4n-10s)+2\\
		&< \ex(n,s) +|X|+|Y|+|C|+1=e(G^+)\,.
	\end{align*}
		
	Hence the graph $G^+$ contains an independent set $Z^+$ of size $s+2$. 
	Clearly $Z^+$ can contain at most $s$ old vertices and, due to the edge $xy$,
	at most two new vertices. 
	So without loss of generality we can assume $Z^+=Z\cup \{x, c\}$
	for some $Z\in V(\cF)$. Now the independence of $Z^+$ in $G^+$ 
	yields $Z\cap X=\vn$ and $Z\cap C = \vn$.
\end{proof}
	
In most applications of this lemma we take $C$ to be an independent set 
of size $s$. We thus arrive at the following structural property of fortresses.
  	 
\begin{cor}\label{cor:n}
	If $(n, s)$ is a minimal counterexample, then the fortress of 
	every extremal graph $G\in \Ex$ has the following property: 
	For all $X, Y,Z\in V(\cF)$ such that $XY\in E(\cF)$ 
		there exists some $T\in V(\cF)$ with $ZT\in E(\cF)$ and 
	either $XT\in E(\cF)$ or $YT\in E(\cF)$ 
	(see Figure \ref{fig:eth3}). \qed
\end{cor}

	\begin{figure}[ht]
	\centering
	\begin{tikzpicture}[scale=1]
	
	\coordinate (a1) at (-3,.5);
	\coordinate (b1) at (-3,-.5);
	\coordinate (c1) at (-2, 0);
	
	\coordinate (a2) at (4,.5);
	\coordinate (b2) at (4,-.5);
	\coordinate (c2) at (5, 0);
	\coordinate (d2) at (4.7, .7);
	
	\coordinate (a3) at (8,.5);
	\coordinate (b3) at (8,-.5);
	\coordinate (c3) at (9, 0);
	\coordinate (d3) at (8.7, .7);
	
	\foreach \i in {1,2,3} \draw [thick] (a\i)--(b\i);
	
	\draw [red, thick] (a2)--(d2)--(c2);
	\draw [red, thick] (b3)--(d3)--(c3);
	
	\foreach \i in {a1,a2,a3,b1,b2,b3,c1,c2,c3,d2,d3}
	\fill (\i) circle (2pt);
	
	\node [left] at (a1) {$X$};
	\node [left] at (b1) {$Y$};
	\node [right] at (c1) {$Z$};
	\node [left] at (a2) {$X$};
	\node [left] at (b2) {$Y$};
	\node [right] at (c2) {$Z$};
	\node [right] at (d2) {$T$};
	\node [left] at (a3) {$X$};
	\node [left] at (b3) {$Y$};
	\node [right] at (c3) {$Z$};
	\node [right] at (d3) {$T$};
	
	\node at (-5,0) {For all};
	\node at (1,0) {there exists either};
	\node at (6.5, 0) {or};
	
	\end{tikzpicture}
	\caption{Vertices $X$, $Y$, $Z$, and $T$ in the fortress $\cF$ of~$G$.}
	\label{fig:eth3}
\end{figure}
 
We shall also encounter an application of Lemma~\ref{cl:1} later 
where $C$ is simply the neighbourhood of a vertex whose degree is 
sufficiently large. At that occasion, it will be useful to know that
only a small number of vertices is not suitable for this purpose. 
  
\begin{lemma}\label{lem:q}
	If $4n/11<s<3n/8$ and $G\in\Ex$, then all but less than $3s-n$ vertices 
	$x\in V(G)$ satisfy $\deg_G(x)> 4n-10s$.
\end{lemma}

\begin{proof}
	Let $q$ denote the number of small-degree vertices under consideration. 
	By summing up the vertex degrees we obtain  
		\begin{align*}
		ns - (3n - 8s)(11s-4n)
		&=
		2g_4(n,s)
		\le
		2e(G) \\
		&\le (n-q)s + q(4n-10s) \\
		&= ns -q(11s - 4n)\,,
	\end{align*}
		whence 
		\[
		q\le 3n-8s<(3n-8s)+(11s-4n)=3s-n\,. \qedhere
	\]
\end{proof} 

Utilising Lemma~\ref{lem:twosets} and Corollary~\ref{cor:n} it is not hard 
to see that for minimal counterexamples $(n, s)$ the fortresses of extremal
graphs have no isolated vertices. A somewhat different argument leads to a
marginally stronger result, the following.  
 
\begin{lemma}\label{cl:62}
	If $(n, s)$ is a minimal counterexample, $G\in \Ex$, and $Q\subseteq V(G)$ 
	is an independent set of size at least $s-1$, then there exists an 
	independent set $X$ of size $s$ with $Q\cap X = \vn$.
\end{lemma}
	
\begin{proof}
	We start by showing that
		\begin{equation}\label{eq:44}
		\ex(n+1, s) < \ex(n,s) + (s-1)\,.
	\end{equation}
		In the special case $\delta \le 2$ the trivial bound~\eqref{eq:trivial} 
	entails indeed
	\begin{align*}
			2\ex(n+1, s) 
			&\le 
			(n+1)s=2 g_4(n, s)+(11s-4n)(3n-8s) +s\\
			&< 
			2\ex(n,s) +\delta(3n-8s)  + s \le 2\ex(n,s) + 2(3n-8s) +s\\
			&= 
			2\ex(n,s) + (4n-11s) + 2(n-3s) +2s < 2\ex(n,s) +2s-2\,.
	\end{align*}
	Otherwise $\delta \ge 3$ and $\delta(n+1, s)=\delta-4<\delta$ combined 
	with the minimality of $(n, s)$ yields
		\begin{align*}
		\ex(n+1, s) 
		&\le 
		g_4(n+1, s) = g_4(n,s) + 12 n -32s+6\\
		&= 
		g_4(n,s) +s +6 -3\delta < \ex(n,s) + s-1\,.
	\end{align*}
	
	Thereby~\eqref{eq:44} is proved. Now we construct a triangle-free 
	graph $G^+$ by adding to $G$ a new vertex $q$ and all edges from $q$ to $Q$.
	Owing to~\eqref{eq:44} we have $e(G^+) > \ex(n+1, s)$ and, therefore, 
	$G^+$ contains an independent set of size $s+1$. Such a set needs to be of 
	the form $X\cup \{q\}$, where $X$ is as required.
\end{proof}

\section{Imprints of the pentagon}

The main result of this section asserts that for minimal 
counterexamples $(n, s)$ the fortresses of extremal graphs $G\in\Ex$ 
cannot contain pentagons. The first idea one might have towards proving this 
is that such a pentagon could lead, due to some general properties of 
fortresses, to a $\G_3$-imprint, which would contradict Corollary~\ref{cor:210}. 
After a quick look at a picture of $\G_3$ and a pentagon $\Gamma_2$ 
contained in it we notice that one of the three extra vertices of $\Gamma_3$
has precisely one neighbour in $V(\Gamma_2)$. This raises the question whether 
it is always true that given a pentagon in a fortress~$\cF$, there exists 
a vertex of~$\cF$ with exactly one neighbour in the pentagon. 
Here is a positive answer for moulds.  

\begin{lemma}\label{lem:62a}
	If $(n, s)$ is a minimal counterexample, $G\in \Ex$, and $(\phi, \psi)$
	denotes a $\G_2$-mould for $G$, then there is a vertex $X\in V(\cF)$ adjacent 
	to exactly one of the vertices $\phi(v)$ with $v\in V(\G_2)$.
\end{lemma}

\begin{proof}
	Put $\varphi(i)=A_i$ and $\psi(i)= B_i$ for every $i\in V(\G_2) = \ZZ/5\ZZ$.
	Recall that $A_i\in V(\cF)$,
	while $B_i$ is an independent set of size $3s-n$ such 
	that $K(A_i, B_i)\subseteq E(G)$. Moreover, a pair $ij\in [5]^{(2)}$
	satisfies $A_iA_j\in E(\cF)$ if and only if $|i-j|\in \{2,3\}$.
 
	Pick arbitrary vertices $b_i\in B_i$ and let $G^-$ be the graph obtained 
	from $G$ by deleting these five vertices $b_i$. 
	Because of $\delta(n-5, s-2)=\delta-2$, the minimality of $(n, s)$, and 
	Fact~\ref{f:use} we have 		
		\begin{align*}
		\ex(n-5, s-2) 
		&\le 
		g_4(n-5, s-2)=g_4(n,s)+4n-16s+6\\
		&<
		\ex(n,s) -\delta - 5s + 6 \le \ex(n,s) - (5s-5) \\
		&=
		e(G) - (5s-5) = e(G^-)\,,
	\end{align*}
		which together with $\nu(G^-)=n-5$ shows $\alpha(G^-)\ge s-1$. 
	Let $Q$ be an independent set in $G^-$ of 
	size $|Q|=\alpha(G^-)\in\{s-1, s\}$ and 
	put $I=\{i\in \ZZ/5\ZZ\colon B_i\cap Q\neq \vn \}$. 
	As $Q^+ = Q\cup \{b_i\colon i\in I\}$ is independent in $G$, 
	we have $|Q|+|I|=|Q^+|\le s$, whence~$|I|\le 1$. 
	
	Suppose first that $|I|=1$, which guarantees $|Q^+|=s$. Moreover, if 
	$i$ denotes the unique element of $I$, Fact~\ref{f:45} tells us that~$Q^+$ 
	is adjacent in~$\cF$ to~$A_i$ but to none of the sets $A_j$ with $j\ne i$, 
	so that~$Q^+$ has the desired property.  
	
	It remains to consider the case $I=\vn$. By Lemma~\ref{cl:62} there exists 
	an independent set $X$ in $G$ of size $s$ with $Q\cap X=\vn$. Each of the 
	sets $B_i$ is either contained in $X$ or disjoint to~$X$ (by Fact~\ref{f:45})
	and, in fact, at most $\alpha(\G_2)=2$ of them can be subsets of $X$.
	Thus~$Q$,~$X$, and three of the sets $B_i$ are mutually disjoint subsets 
	of $V(G)$, wherefore
		\[
		n\ge|Q|+|X|+3(3s-n) \ge (s-1) + s + (9s-3n) = n+\delta-1\ge n\,.
	\]
		Now equality needs to hold throughout, which means that 
		\[
		\alpha(G^-)=|Q|=s-1
		\quad \text{  and } \quad 11s-4n=\delta=1\,.
	\]
		In particular, $s$ is odd.
	Next we look at the set 
		\[
		C = \{v\in V(G)\colon |\Nn_G(v)\cap \{b_1, \dots, b_5\}|\le 1\}\,.
	\]
		
	Since every vertex $v\in V(G)\setminus C$ belongs 
	to exactly two of the sets $A_1,\dots, A_5$, 
	we have $|V(G)\setminus C|\le \lfloor 5s/2\rfloor=(5s-1)/2=(2n-3s)$
	and, hence, $|C| \ge n-(2n-3s)=3s-n$. 
	Thus Lemma~\ref{lem:q} yields some $c\in C$ 
	such that $\deg_G(c)\ge 4n-10s+1=s$.
	Now $Z=\Nn_G(c)$ is an independent set of size $s$. Due to $\alpha(G^-)=s-1$
	it needs to contain at least one the vertices $b_i$ and combined with 
	the definition of $C$ this shows that there is a unique vertex 
	in $Z\cap \{b_1,\dots, b_5\}$. If $b_i$ denotes that vertex, then 
	$\Nn_\cF(Z)\cap \{A_1, A_2, A_3, A_4, A_5\}=\{A_i\}$ and, consequently,~$Z$
	is as desired.
\end{proof}

Before going any further we state two more facts on moulds for graphs 
in $\Ex$, both of which rely on space limitations caused by $s>\frac4{11}n$.

\begin{fact}\label{f:edge}
	 Let $G\in \Ex$ be an extremal graph, where $s>\frac 4{11}n$.
	 Suppose further that~$(\varphi,\psi)$ is an $H$-mould in $G$
	 for some graph $H$. 
	 If $XY$ denotes an arbitrary edge of $\cF$ and $i, j, k\in V(H)$ 
	 are distinct, then some edge of $\cF$ connects a vertex from  
	 $\{X,Y\}$ with a vertex from $\{\varphi(i), \varphi(j), \varphi(k)\}$. 
\end{fact}
\begin{proof}
	In view of 
		\[
		|X|+|Y|+|\psi(i)|+|\psi(j)|+|\psi(k)| = 2s + 3(3s-n) = 11s-3n > n
	\]
		the five sets $X$, $Y$, $\psi(i)$, $\psi(j)$, $\psi(k)$ cannot be mutually 
	disjoint. But $XY\in E(\cF)$ means that the first two sets are disjoint and 
	Fact~\ref{f:00} informs us that the last three sets are disjoint. 
	So without loss of generality $X\cap \psi(i)\ne\vn$, which due to 
	Fact~\ref{f:45} implies $X\phi(i)\in E(\cF)$.
\end{proof}

\begin{fact}\label{f:claw}
	If $s>\frac4{11}n$ and $(\phi, \psi)$ denotes a $K_{1,3}$-mould in 
	some graph $G\in \Ex$, then $\{\phi(v)\colon v\in V(K_{1, 3})\}$ dominates $V(\cF)$.
\end{fact}

\begin{proof}
	Let $\{x\}$ and $\{u, v, w\}$ be the vertex classes of $K_{1, 3}$. 
	Consider an arbitrary fortress vertex $Z\in V(\cF)$ that is non-adjacent 
	to $\phi(x)$. We need to prove that at least one 
	of $\phi(u)$, $\phi(v)$, or $\phi(w)$ is a neighbour of $Z$.
	
	Fact~\ref{f:45} tells us $Z\cap \psi(x)=\vn$ and together with the 
	independence of $(Z\setminus \phi(x))\cup \psi(x)$ this yields 
	$|Z\cap \phi(x)|\ge |\psi(x)|\ge 3s-n$. Now $Z\cap \phi(x)$, $\psi(u)$, 
	$\psi(v)$, and $\psi(w)$ are four subsets of $\phi(x)$ whose cardinalities 
	sum up to at least $4(3s-n)>s$. Hence, these sets cannot be mutually disjoint
	and Fact~\ref{f:00} allows us to suppose, without loss of generality, that
	$Z\cap \psi(u)\ne\vn$. But now $Z\phi(u)\in E(\cF)$ is immediate from  
	Fact~\ref{f:45}. 
\end{proof}

The final result of this section improves upon Corollary~\ref{cor:210}. 

\begin{lemma}\label{lem:61}
	If $(n, s)$ is a minimal counterexample, then no graph in $\Ex$ contains 
	a $\G_2$-imprint.  
\end{lemma}

\begin{proof} 	Throughout the argument we shall encounter four  
	graphs $G_0, G_1, G_2, G_3\in \Ex$ and without further notice we shall 
	always denote the fortress of $G_j$ by $\cF_j$. Roughly speaking, our 
	strategy is to show that starting with a $\Gamma_2$-mould one can obtain 
	a $\Gamma_3$-imprint by means of two symmetrisation steps. 
	
	\smallskip
	
	\noindent {\bf Stage A: The graph $G_0$.} 
	Assuming that our claim fails Lemma~\ref{lem:43} yields a 
	graph $G_0\in \Ex$ containing a $\G_2$-mould $(\varphi, \psi)$. 
	Set $A_i=\varphi(i)$ and $B_i=\psi(i)$ for every $i\in V(\G_2)$.
	By Lemma~\ref{lem:62a} some vertex $A_6\in V(\cF_0)$ is adjacent 
	to exactly one of the sets~$A_i$, say to~$A_3$. Now 
	Figure~\ref{fig:1-6a} shows an induced subgraph of $\cF_0$.
	Due to Fact~\ref{f:45} we have $B_3\subseteq A_6$, while 
	$B_1$, $B_2$, $B_4$, $B_5$ are disjoint to $A_6$. 	
	
			\begin{figure}[ht]
	\centering
	
	\begin{subfigure}[b]{0.32\textwidth}
		\centering

	\begin{tikzpicture}[scale=.6]
		
		\coordinate (a1) at (90:2);
		\coordinate (a5) at (162:2);
		\coordinate (a4) at (234:2);
		\coordinate (a3) at (306:2);
		\coordinate (a2) at (18:2);
		\coordinate (a6) at (-20:3.5);
		\coordinate (a7) at (18:3);
		\coordinate (a8) at (225:3.5);

		\draw (a1)--(a3)--(a5)--(a2)--(a4)--cycle;
				\draw (a3)--(a6);
		
		\foreach \i in {1,...,5}\draw [fill=red!50, thick] (a\i) circle (5pt);
		\fill (a6) circle (3pt);
		
		\node [left] at (a1) {$A_1$};
		\node [right] at (a2) {$A_2$};
		\node [left] at (a3) {$A_3$};
		\node [left] at (a4) {$A_4$};
		\node [left] at (a5) {$A_5$};
		\node at ($(a6)+(.45,-.1)$) {$A_6$};

	\end{tikzpicture}
	
	\caption{}
		\label{fig:1-6a}
	
\end{subfigure}
\hfill    
\begin{subfigure}[b]{0.32\textwidth}
	\centering
	
	\begin{tikzpicture}[scale=.6]
	
	\coordinate (a1) at (90:2);
	\coordinate (a5) at (162:2);
	\coordinate (a4) at (234:2);
	\coordinate (a3) at (306:2);
	\coordinate (a2) at (18:2);
	\coordinate (a6) at (-20:3.5);
	\coordinate (a7) at (25:3.6);
	\coordinate (a8) at (225:3.5);

	\fill (a7) circle (3pt);
	
	\draw (a1)--(a3)--(a5)--(a2)--(a4)--cycle;
	\draw (a6)--(a7);
		\draw (a3)--(a6);
	\draw (a2)--(a7);
	
	\draw [blue, thick] (a2)--(a4);
	\foreach \i in {1,...,5}\draw [fill=red!50, thick] (a\i) circle (5pt);
	\fill [blue] (a6) circle (3.5pt);
	
	\node [left] at (a1) {$A_1$};
	\node at ($(a2)+(.45,-.2)$) {$A_2$};
	\node [left] at (a3) {$A_3$};
	\node [left] at (a4) {$A_4$};
	\node [left] at (a5) {$A_5$};
	\node at ($(a6)+(.45,-.1)$) {$A_6$};
\node at ($(a7)+(.45,-.1)$) {$A_7$};
		
		\end{tikzpicture}
	
	\caption{}
	\label{fig:1-6b}

\end{subfigure}
\hfill    
\begin{subfigure}[b]{0.32\textwidth}
	\centering

	\begin{tikzpicture}[scale=.6]
	
	\coordinate (a1) at (90:2);
	\coordinate (a5) at (162:2);
	\coordinate (a4) at (234:2);
	\coordinate (a3) at (306:2);
	\coordinate (a2) at (18:2);
	\coordinate (a6) at (-20:3.5);
	\coordinate (a7) at (25:3.6);
	\coordinate (a8) at (225:3.5);
	
	\draw (a1)--(a3)--(a5)--(a2)--(a4)--cycle;
	\draw (a6)--(a7)--(a1);
		\draw (a3)--(a6);
	\draw (a2)--(a7);
	
		\foreach \i in {1,...,7}\fill (a\i) circle (3pt);
		\foreach \i in {1,...,5}\draw [fill=red!50, thick] (a\i) circle (5pt);
	
	\node [left] at (a1) {$A_1$};
	\node at ($(a2)+(.45,-.2)$) {$A_2$};
	\node [left] at (a3) {$A_3$};
	\node [left] at (a4) {$A_4$};
	\node [left] at (a5) {$A_5$};
	\node at ($(a6)+(.45,-.1)$) {$A_6$};
	\node at ($(a7)+(.45,-.1)$) {$A_7$};
		
	\end{tikzpicture}
	
	\caption{}
	\label{fig:1-6c}
	
	\end{subfigure}
	\caption{Extending a $\G_2$-mould in $\cF_0$. Large \re{red} dots indicate 
	vertices~$A_i$ of $\cF_0$ for which there is a set $B_i$ such that $|B_i|=3s-n$  
	and $K(A_i,B_i)\subseteq E(G)$.}
	\label{fig:1-6t}
\end{figure}
 	
	By Corollary~\ref{cor:n} applied to the edge $A_2A_4$ and the vertex $A_6$ 
	there exists some $A_7\in V(\cF_0)$ such that, without loss of generality, 
	$A_2A_7$ and $A_6A_7$ are edges of $\cF_0$ (see Figure \ref{fig:1-6b}). 
	Next we apply Fact~\ref{f:edge} to the edge $A_6A_7$ and the vertices 
	$A_1$, $A_4$, $A_5$ of $\cF_0$. Since $\cF_0$ is triangle-free, the only 
	possible outcome is $A_1A_7\in E(\cF_0)$ and Figure~\ref{fig:1-6c}
	shows an induced subgraph of~$\cF_0$. 
	
	\smallskip
	
	\noindent {\bf Stage B: The graph $G_1$.}
	Corollary~\ref{lem:B'} leads to the existence of a set $B_7\subseteq A_6$ of 
	size $|B_7|=3s-n$ such that $G_1=\GZS(G_0|A_7, B_7)$ is again in $\Ex$. 
	It turns out that everything we know about $\cF_0$ remains valid for $\cF_1$.
	In particular, $A_6$ is still independent in $G_1$. Moreover, we contend that
	\begin{equation}\label{eq:G1}
			K(A_i, B_i) \subseteq E(G_1)
			\quad \text{ for every $i\in [7]\setminus \{6\}$}\,,
	\end{equation}
	which in turn implies that $A_i$ and $B_i$ are independent in $G_1$ as well.	
	The inclusion~\eqref{eq:G1} is completely clear for $i=7$. 
	The edges $A_1A_7$, $A_2A_7$ show $B_7\subseteq A_1\cap A_2$
	as well as $(B_1\cup B_2)\subseteq A_7$, and thus~\eqref{eq:G1}
	holds for $i=1, 2$ as well. 
	
	In order to take care of the three remaining cases, it suffices to 
	show $B_7\cap (A_i\cup B_i)=\vn$ for $i=3, 4, 5$. Since $B_7$ is contained
	in $A_1$, it is indeed disjoint to $A_3$, $A_4$, and $B_5$ and, similarly, 
	$B_7\subseteq A_2$ yields $B_7\cap A_5=B_7\cap B_3=\vn$. 
	Finally, $B_7\subseteq A_6$ implies $B_7\cap B_4=\vn$.  
	
		\begin{figure}[ht]
	\centering
	
	\begin{tikzpicture}[scale=.7]
	
	\coordinate (a1) at (90:2);
	\coordinate (a5) at (162:2);
	\coordinate (a4) at (234:2);
	\coordinate (a3) at (306:2);
	\coordinate (a2) at (18:2);
	\coordinate (a6) at (-20:3.5);
	\coordinate (a7) at (25:3.6);
	\coordinate (a8) at (225:3.5);

	\foreach \i in {1,...,7}\fill (a\i) circle (2pt);
	
	\draw (a1)--(a3)--(a5)--(a2)--(a4)--cycle;
	\draw (a6)--(a7)--(a1);
		\draw (a3)--(a6);
	\draw (a2)--(a7);
	
		\foreach \i in {1,...,7}\fill (a\i) circle (3pt);
		\foreach \i in {1,...,5,7}\draw [fill=red!50, thick] (a\i) circle (5pt);
	
	\node [left] at (a1) {$A_1$};
	\node at ($(a2)+(.45,-.2)$) {$A_2$};
	\node [left] at (a3) {$A_3$};
	\node [left] at (a4) {$A_4$};
	\node [left] at (a5) {$A_5$};
	\node at ($(a6)+(.45,-.1)$) {$A_6$};
\node at ($(a7)+(.45,-.2)$) {$A_7$};
		
	\end{tikzpicture}

	\caption{A common subgraph of $\cF_1$ and $\cF_3$.
	}
	\label{fig:2}
\end{figure}
 	
	\noindent {\bf Stage C: The graphs $G_2$ and $G_3$.}
	Let us now consider a matching $M$ between $A_7$ 
	and $V(G_1)\setminus(A_6\cup A_7)$ such that 	
		\begin{enumerate}
		\item[$\bullet$] $|E(M)|$ is maximum;
		\item[$\bullet$] and, subject to this, $|V(M)\cap B_1|$ is maximum.
	\end{enumerate}
		Using $B_1\subseteq A_7$, $B_4\subseteq V(G_1)\setminus(A_6\cup A_7)$,
	and $K(B_1, B_4)\subseteq E(G_1)$ one can easily see that~$M$ covers $B_1$. 
	By Lemma~\ref{lem:B''}, if $B_6\subseteq A_7\setminus V(M)$ denotes any set of 
	size $3s-n$, then $G_2=\GZS(G_1| A_6, B_6)$ is again in $\Ex$.
	Arguing as in the previous stage, one can show that with the possible 
	exception of $A_4$ everything we know about $\cF_1$ survives in $\cF_2$. 
	More precisely, we shall show 
	 
	\begin{equation}\label{eq:G2}
			K(A_i, B_i) \subseteq E(G_2)
			\quad \text{ for every $i\in [7]\setminus \{4\}$}\,,
	\end{equation}
	for which reasons these sets $A_i$, $B_i$ are still independent in $G_2$. 
	
	\begin{figure}[ht]
	\centering
	
	\begin{tikzpicture}[scale=.7]
	
	\coordinate (a1) at (90:2);
	\coordinate (a5) at (162:2);
	\coordinate (a4) at (234:2);
	\coordinate (a3) at (306:2);
	\coordinate (a2) at (18:2);
	\coordinate (a6) at (-20:3.5);
	\coordinate (a7) at (25:3.6);
	\coordinate (a8) at (225:3.5);

	\draw (a1)--(a3)--(a5)--(a2);
	\draw (a6)--(a7)--(a1);
		\draw (a3)--(a6);
	\draw (a2)--(a7);

		\foreach \i in {1,...,3,5,6,7}\draw [fill=red!50, thick] (a\i) circle (5pt);
	
	\node [left] at (a1) {$A_1$};
	\node at ($(a2)+(.45,-.2)$) {$A_2$};
	\node [left] at (a3) {$A_3$};
	\node [left] at (a5) {$A_5$};
	\node at ($(a6)+(.45,-.2)$) {$A_6$};
\node at ($(a7)+(.45,-.2)$) {$A_7$};
		
	\end{tikzpicture}

	\caption{A subgraph of $\cF_2$.}
	\label{fig:3}
\end{figure}

	This time, the case $i=6$ is obvious and the edges $A_3A_6$, $A_6A_7$
	take care of the cases $i=3, 7$. 
		\begin{center}
	\begin{tabular}{c|c}
		Furthermore & shows that $B_6$ is disjoint to \\
		\hline
		$B_6\subseteq A_3$ &  $A_1$, $A_5$, $B_2$\\
		$B_6\subseteq A_7$ &  $A_2$, $B_4$, $B_5$.  
	\end{tabular}
	\end{center}
		Since our special choice of $M$ enforces $B_1\cap B_6=\vn$ we thus 
	have $B_6\cap (A_i\cup B_i)=\vn$ for $i=1, 2, 5$ and the 
	remaining three cases of~\eqref{eq:G2} have thereby been proved as well.

	The loss of control over $A_4$ is a severe problem. 
	We address this situation by setting $\wt{B}_6=B_6\sm A_4$
	and working mainly with the graph $G_3=\GZS(G_1|A_6, \wt{B}_6)$,
	which by Lemma~\ref{lem:B''} belongs to $\Ex$, too. 
	Due to $\wt{B}_6\cap (A_4\cup B_4)=\vn$ the set $A_4$ is still 
	independent in $G_3$ and the usual arguments show
		\begin{equation}\label{eq:G3}
			K(A_i, B_i) \subseteq E(G_3)
			\quad \text{ for every $i\in [7]\setminus \{6\}$}.
	\end{equation}
		
	So, roughly speaking, $G_3$ inherits all useful properties of $G_1$. 
	In fact, at this moment it is not straightforward to say what the advantage 
	of $G_3$ over $G_1$ is, or whether these graphs are actually distinct 
	from one another. This will only become apparent towards the very end 
	of the proof.  
	
	\smallskip
	
	\noindent {\bf Stage D: The eighth vertex.} 
	Now $A_1A_7A_2A_5A_3$ is a pentagon in $\cF_3$ and together with the 
	sets $B_1$, $B_7$, $B_2$, $B_5$, $B_3$ it forms a $\G_2$-mould for $G_3$
	(see Figure~\ref{fig:1-6d}).
	
		\begin{figure}[ht]
	\centering
	
		\begin{subfigure}[b]{0.49\textwidth}
		\centering

	\begin{tikzpicture}[scale=.7]
		
		\coordinate (a1) at (90:2);
		\coordinate (a5) at (162:2);
		\coordinate (a4) at (234:2);
		\coordinate (a3) at (306:2);
		\coordinate (a2) at (18:2);
		\coordinate (a6) at (-20:3.5);
			\coordinate (a7) at (25:3.6);
		\coordinate (a8) at (225:3.5);

		\foreach \i in {1,...,7}\fill (a\i) circle (3pt);
		
		\draw (a1)--(a3)--(a5)--(a2)--(a4)--cycle;
		\draw (a6)--(a7)--(a1);
			\draw (a3)--(a6);
		\draw (a2)--(a7);
		
			\draw[blue, thick] (a1)--(a3)--(a5)--(a2)--(a7)--cycle;

			\foreach \i in {1, 2, 3, 4, 5, 7}[red]\draw [fill=red!50, thick] (a\i) circle (5pt);

		\node [left] at (a1) {$A_1$};
		\node at ($(a2)+(.45,-.2)$) {$A_2$};
		\node [left] at (a3) {$A_3$};
		\node [left] at (a4) {$A_4$};
		\node [left] at (a5) {$A_5$};
			\node at ($(a6)+(.45,-.1)$) {$A_6$};
		\node at ($(a7)+(.45,-.2)$) {$A_7$};

	\end{tikzpicture}

\end{subfigure}
\hfill    
\begin{subfigure}[b]{0.49\textwidth}
\centering
\begin{tikzpicture}[scale=.6]
	
	\coordinate (a1) at (90:2.5);
	\coordinate (a5) at (162:2.5);
	\coordinate (a4) at (234:2.5);
	\coordinate (a3) at (306:2.5);
	\coordinate (a2) at (18:2.5);
	\coordinate (a6) at (140:.8);
	\coordinate (a7) at (-30:.5);
	\coordinate (a8) at (189:3.5);

	\foreach \i in {1,...,7}\fill (a\i) circle (3pt);

	\draw (a5)--(a6)--(a2);
	\draw (a1)--(a7) -- (a3);

	\draw[blue, thick] (a1)--(a5)--(a4)--(a3)--(a2)--cycle;

		\foreach \i in {1, 2, 3, 4, 5, 7}\draw [fill=red!50, thick] (a\i) circle (5pt);

	\node  at ($(a1)+(.6,.05)$) {$A_1$};
	\node [right] at (a2) {$A_7$};
	\node [right] at (a3) {$A_2$};
	\node [left] at (a4) {$A_5$};
	\node [left] at (a5) {$A_3$};
	\node  at ($(a6)+(.1,.4)$){$A_6$};
	\node [right] at (a7) {$A_4$};
			
\end{tikzpicture}
\end{subfigure}
	\caption{$\G_2$-mould in $\cF_3$ (in \bl{blue}).}
	\label{fig:1-6d}
\end{figure} 	
	Thus Lemma~\ref{lem:62a} yields a vertex $A_8\in V(\cF_3)$ adjacent 
	to exactly one of $A_1$, $A_2$, $A_3$, $A_7$, or $A_5$. 
	If this unique neighbour of $A_8$ is $A_1$, then Fact~\ref{f:claw}
	applied to the claw in Figure~\ref{fig:4A} yields $A_4A_8\in E(\cF_3)$ and thus 
	$A_1A_4A_8$ is a triangle in $\cF_3$, which is absurd. 
	Working with Figure~\ref{fig:4B} a similar contradiction can be obtained 
	if $A_2A_8\in E(\cF_3)$. 	
	Suppose next that $A_3A_8\in E(\cF_3)$ or $A_7A_8\in E(\cF_3)$. 
	Due to $B_6\subseteq A_3\cap A_7$ we then have $B_6\cap A_8=\vn$ 
	and, therefore, $A_8$ remains independent in $G_2$. 
	So we can apply Fact~\ref{f:claw} in $\cF_2$ to $A_8$ and the claws drawn
	in Figures~\ref{fig:4C} and~\ref{fig:4D}, thus reaching the same 
	contradiction as before. 
	Summarising this discussion, $A_8$ needs to be adjacent to $A_5$
	and none of $A_1$, $A_2$, $A_3$, $A_7$ (see Figure \ref{fig:1-6ea}). 
	
	\begin{figure}[ht]
\centering
	
\begin{subfigure}[b]{0.22\textwidth}
	\centering
	
	\begin{tikzpicture}[scale=.6]
	
	\coordinate (a1) at (-2,2);
	\coordinate (a2) at (0,2);
	\coordinate (a3) at (2,2);
	\coordinate (a4) at (0,0);
	
	\draw (a1)--(a4)--(a2)--(a4)--(a3);
	
	\foreach \i in {1,...,4}\draw [fill=red!50, thick] (a\i) circle (5pt);
	
	\node [above] at (a1) {$A_4$};
	\node [above] at (a2) {$A_5$};
	\node [above] at (a3) {$A_7$};
	\node at ($(a4)+(0,-.7)$) {$A_2$};	
	\end{tikzpicture}
\caption{}
\label{fig:4A}
\end{subfigure}
\hfill    
\begin{subfigure}[b]{0.22\textwidth}
	\centering
	
	\begin{tikzpicture}[scale=.6]
		
		\coordinate (a1) at (-2,2);
		\coordinate (a2) at (0,2);
		\coordinate (a3) at (2,2);
		\coordinate (a4) at (0,0);
		
		\draw (a1)--(a4)--(a2)--(a4)--(a3);
		
		\foreach \i in {1,...,4}\draw [fill=red!50, thick] (a\i) circle (5pt);
		
		\node [above] at (a1) {$A_3$};
		\node [above] at (a2) {$A_4$};
		\node [above] at (a3) {$A_7$};
		\node at ($(a4)+(0,-.7)$) {$A_1$};
		
	\end{tikzpicture}
	\caption{}
	\label{fig:4B}	
\end{subfigure}
\hfill
\begin{subfigure}[b]{0.22\textwidth}
	\centering
	
	\begin{tikzpicture}[scale=.6]
	
	\coordinate (a1) at (-2,2);
	\coordinate (a2) at (0,2);
	\coordinate (a3) at (2,2);
	\coordinate (a4) at (0,0);
	
	\draw (a1)--(a4)--(a2)--(a4)--(a3);
	
	\foreach \i in {1,...,4}\draw [fill=red!50, thick] (a\i) circle (5pt);
	
	\node [above] at (a1) {$A_1$};
	\node [above] at (a2) {$A_2$};
	\node [above] at (a3) {$A_6$};
	\node at ($(a4)+(0,-.7)$) {$A_7$};
	
	\end{tikzpicture}
\caption{}
	\label{fig:4C}
\end{subfigure}
\hfill    
\begin{subfigure}[b]{0.22\textwidth}
	\centering
	
	\begin{tikzpicture}[scale=.6]
		
		\coordinate (a1) at (-2,2);
		\coordinate (a2) at (0,2);
		\coordinate (a3) at (2,2);
		\coordinate (a4) at (0,0);
		
		\draw (a1)--(a4)--(a2)--(a4)--(a3);
		
		\foreach \i in {1,...,4}\draw [fill=red!50, thick] (a\i) circle (5pt);
		
		\node [above] at (a1) {$A_1$};
		\node [above] at (a2) {$A_5$};
		\node [above] at (a3) {$A_6$};
		\node at ($(a4)+(0,-.7)$) {$A_3$};
		
	\end{tikzpicture}
	\caption{}
	\label{fig:4D}
	\end{subfigure}

	\caption{Claw-moulds in $\cF_3$, $\cF_3$, $\cF_2$, and $\cF_2$.
	}
	\label{fig:4}
\end{figure} 	
	\noindent {\bf Stage E: The last two edges.}
	Returning to $\cF_3$ we apply Fact~\ref{f:edge} to
	the edge $A_5A_8$ and to the vertices $A_1$, $A_4$, $A_7$, 
	thus learning $A_4A_8\in E(\cF_3)$ (see Figure \ref{fig:1-6eb}). 
	Assume for the sake of contradiction that $B_6\cap A_8=\vn$. 
	This means that $A_8$
	stays independent in $G_2$ and by Fact~\ref{f:edge} applied in $\cF_2$
	to the edge $A_5A_8$ and the vertices $A_1$, $A_6$, $A_7$ we obtain
	$A_6A_8\in E(\cF_2)$, which in turn yields the 
	contradiction $B_6\subseteq A_8$.

	We have thereby shown $B_6\cap A_8\ne\vn$ 
	and together with $A_4\cap A_8=\vn$
	we reach $\wt{B}_6\cap A_8\ne\vn$. Thus the definition of $G_3$ reveals 
	$A_6A_8\in E(\cF_3)$ (see Figure \ref{fig:1-6ec}). 
	
	Altogether our eight vertices $A_i$ form an imprint of $\G_3$ in $\cF_3$ 
	with natural ordering $A_1A_2A_8A_3A_7A_4A_5A_6$, 
	contrary to Corollary~\ref{cor:210}. 
\end{proof}	

		\begin{figure}[ht]
	\centering

	\begin{subfigure}[b]{0.32\textwidth}
		\centering

	\begin{tikzpicture}[scale=.6]
		
	\coordinate (a1) at (90:2.5);
	\coordinate (a5) at (162:2.5);
	\coordinate (a4) at (234:2.5);
	\coordinate (a3) at (306:2.5);
	\coordinate (a2) at (18:2.5);
	\coordinate (a6) at (140:.8);
	\coordinate (a7) at (-30:.5);
	\coordinate (a8) at (220:.7);
	
	\foreach \i in {1,...,8}\fill (a\i) circle (3pt);
	
	\draw (a5)--(a6)--(a2);
	\draw (a1)--(a7) -- (a3);
	\draw (a8) -- (a4);

	\draw[blue, thick] (a1)--(a5)--(a4)--(a3)--(a2)--cycle;

		\foreach \i in {1, 2, 3, 4, 5, 7}\draw [fill=red!50, thick] (a\i) circle (5pt);

	\node  at ($(a1)+(.6,.05)$) {$A_1$};
\node [right] at (a2) {$A_7$};
\node [right] at (a3) {$A_2$};
\node [left] at (a4) {$A_5$};
\node [left] at (a5) {$A_3$};
\node  at ($(a6)+(.1,.4)$){$A_6$};
\node [right] at (a7) {$A_4$};
	\node [left] at (a8) {$A_8$};

	\end{tikzpicture}
\caption{}
\label{fig:1-6ea}

\end{subfigure}
\hfill    
\begin{subfigure}[b]{0.32\textwidth}
\centering
\begin{tikzpicture}[scale=.6]
	
	\coordinate (a1) at (90:2.5);
	\coordinate (a5) at (162:2.5);
	\coordinate (a4) at (234:2.5);
	\coordinate (a3) at (306:2.5);
	\coordinate (a2) at (18:2.5);
	\coordinate (a6) at (140:.8);
	\coordinate (a7) at (-30:.5);
	\coordinate (a8) at (220:.7);

	\foreach \i in {1,...,8}\fill (a\i) circle (3pt);

	\draw (a5)--(a6)--(a2);
	\draw (a1)--(a7) -- (a3);
	\draw (a7) -- (a8) -- (a4);

	\draw[blue, thick] (a1)--(a5)--(a4)--(a3)--(a2)--cycle;

		\foreach \i in {1, 2, 3, 4, 5, 7}\draw [fill=red!50, thick] (a\i) circle (5pt);
\node  at ($(a1)+(.6,.05)$) {$A_1$};
\node [right] at (a2) {$A_7$};
\node [right] at (a3) {$A_2$};
\node [left] at (a4) {$A_5$};
\node [left] at (a5) {$A_3$};
\node  at ($(a6)+(.1,.4)$){$A_6$};
\node [right] at (a7) {$A_4$};
\node [left] at (a8) {$A_8$};
\end{tikzpicture}
\caption{}
\label{fig:1-6eb}
\end{subfigure}
\hfill    
\begin{subfigure}[b]{0.32\textwidth}
\centering
\begin{tikzpicture}[scale=.6]
\coordinate (a1) at (90:2.5);
\coordinate (a5) at (162:2.5);
\coordinate (a4) at (234:2.5);
\coordinate (a3) at (306:2.5);
\coordinate (a2) at (18:2.5);
\coordinate (a6) at (140:.8);
\coordinate (a7) at (-30:.5);
\coordinate (a8) at (220:.7);
\foreach \i in {1,...,8}\fill (a\i) circle (3pt);
\draw (a5)--(a6)--(a2);
\draw (a1)--(a7) -- (a3);
\draw (a4) -- (a8) -- (a6);
\draw (a8) -- (a7);
\draw[blue, thick] (a1)--(a5)--(a4)--(a3)--(a2)--cycle;
\foreach \i in {1, 2, 3, 4, 5, 7}\draw [fill=red!50, thick] (a\i) circle (5pt);
\node  at ($(a1)+(.6,.05)$) {$A_1$};
\node [right] at (a2) {$A_7$};
\node [right] at (a3) {$A_2$};
\node [left] at (a4) {$A_5$};
\node [left] at (a5) {$A_3$};
\node  at ($(a6)+(.1,.4)$){$A_6$};
\node [right] at (a7) {$A_4$};
\node [left] at (a8) {$A_8$};
\end{tikzpicture}
\caption{}
\label{fig:1-6ec}
\end{subfigure}
\caption{Construction of a $\G_3$-imprint in $\cF_3$. }
\label{fig:1-6e}
\end{figure}
	 
\section{Bipartite fortresses}

The next result supersedes the previous section. 

\begin{lemma}
	If $(n, s)$ is a minimal counterexample, then all graphs in $\Ex$
	have bipartite fortresses.
\end{lemma}

\begin{proof}
	Assume contrariwise that the fortress $\cF$ of some $G\in \Ex$
	contains an odd cycle. Let $t$ be minimal such that $\cF$ contains 
	a cycle $\cC=A_1A_2\dots A_{2t+1}$ of length $2t+1$. 
	Due to Fact~\ref{fact:c3} and Lemma~\ref{lem:61} we have $t\ge 3$. 
		
	\begin{figure}[ht]
	\centering
	\begin{tikzpicture}[scale=0.62]
		
		\coordinate (a1) at (-8.5,0);
		\coordinate (a2) at (-7.5,-1.1);
		\coordinate (a3) at (-5.6,-1.5);
		\coordinate (a4) at (-3.7,-1.1);
		\coordinate (a5) at (-2.7,0);
		\coordinate (a6) at (-3.7, 1.1);
		\coordinate (a7) at (-5.6, 1.5);
		\coordinate (a8) at (-7.5, 1.1);
		\coordinate (x) at (-5.6,.4);
		
		\draw [line width =4pt, blue!25!white, rounded corners] (a8)--(a1)--(x)--(a5)--(a6) [out=150, in=30] to (a8);
						
						\node at (-5.42, 1.63) {\Huge $\dots$};
			
		\draw [thick] (a8)--(a1)--(a2)--(a3)--(a4)--(a5)--(a6);		\draw [thick] (a1) --(-5.6,.4) -- (a5);
		
			\draw [thick, red] (a1)--(a2);
			
				\fill (x) circle (3pt);
				\node at (-5.7,-.1) {$T$};
				
					\foreach \i in {1,...,6, 8}{
						\fill (a\i) circle (3pt);
					}
				\foreach \i in {1,2,5}
				\fill [red] (a\i) circle (3pt);
		
		\node at ($(a1)+(-.3,-.2)$) {$A_1$};
		\node at ($(a2)+(-.3,-.3)$) {$A_2$};
		\node at ($(a3)+(-.3,-.3)$) {$A_3$};
		\node at ($(a4)+(.3,-.3)$) {$A_4$};
		\node at ($(a5)+(.3,-.3)$) {$A_5$};
		\node at ($(a6)+(.3,.4)$) {$A_6$};
		\node at ($(a8)+(-.6,.4)$) {$A_{2t+1}$};
		
		\coordinate (b5) at (8.5,0);
		\coordinate (b4) at (7.5,-1.1);
		\coordinate (b3) at (5.6,-1.5);
		\coordinate (b2) at (3.7,-1.1);
		\coordinate (b1) at (2.7,0);
		\coordinate (b6) at (7.5, 1.1);
		\coordinate (b7) at (5.6, 1.5);
		\coordinate (b8) at (3.7, 1.1);

	\draw [line width =4pt, blue!25!white, rounded corners] (b2)--(5.6,.4)--(b5)--(b4)--(b3)--(b2);

		\fill (5.6,.4) circle (3pt);
		\node at (5.8,-.1) {$T$};
		\node at (5.6, 1.5) {\Huge $\dots$};
		
		\draw [thick] (b8)--(b1)--(b2)--(b3)--(b4)--(b5)--(b6);
		\draw [thick] (b2) --(5.6,.4) -- (b5);
		
		\draw [thick, red] (b1)--(b2);
		
			\foreach \i in {1,...,6, 8}{
				\fill (b\i) circle (3pt);
			}
			
			\foreach \i in {1,2,5}
			\fill [red] (b\i) circle (3pt);
		
		
			\node at ($(b1)+(-.3,-.2)$) {$A_1$};
		\node at ($(b2)+(-.3,-.3)$) {$A_2$};
		\node at ($(b3)+(-.3,-.3)$) {$A_3$};
		\node at ($(b4)+(.3,-.3)$) {$A_4$};
		\node at ($(b5)+(.3,-.3)$) {$A_5$};
		\node at ($(b6)+(.3,.4)$) {$A_6$};
		\node at ($(b8)+(-.6,.4)$) {$A_{2t+1}$};

	\end{tikzpicture}
	\caption{The cycle $\cC=A_1A_2\dots A_{2t+1}$ and the vertex $T$.}
	\label{fig:oddC}
\end{figure}

	Now Corollary~\ref{cor:n} applied to the edge $A_1A_2$ and the vertex $A_5$ 
	yields some $T\in V(\cF)$ such that $A_5T\in E(\cF)$ and 
	either $A_1T\in E(\cF)$ or $A_2T\in E(\cF)$. But this creates a closed walk
	in $\cF$, which passes through $T$ and has length $2t-1$ or $5$ 
	(see Figure~\ref{fig:oddC}). In both cases we reach a contradiction 
	to the minimality of $t$.
\end{proof}

The reason why being bipartite is a useful property of fortresses is that 
it leads to the existence of an edge with a remarkable property. 

\begin{lemma}\label{lem:56}
	For every minimal counterexample $(n, s)$ and every graph
	$G\in \Ex$ there exists an edge $XY\in E(\cF)$ such that 
	every $Z\in V(\cF)$ satisfies $|X\cap Z|<3s-n$ or $|Y\cap Z|<3s-n$.
\end{lemma}

\begin{proof}	The previous lemma informs us that $\cF$ is bipartite, say with vertex 
	classes 	$\cX$ and~$\cY$. Since $\cF$ has at least one edge 
	(by Lemma~\ref{lem:twosets}), these classes cannot be empty.
	The minimality of $(n, s)$ is used as follows.  

	\begin{claim}\label{cl:2}
		There do not exist an integer $r\ge 2$, fortress vertices 
		$X_1,X_2,\dots, X_{r}\in\cX$, $Y_1,Y_2,\dots, Y_{r}\in\cY$, 
		and independent sets $C_1,C_2,\dots, C_{r}\subseteq V(G)$ 
		such that for every $i\in \ZZ/r\ZZ$ we have 
			\begin{enumerate}[label=\rmlabel]
				\item\label{it:i1} $X_{i}Y_{i}\in E(\cF)$; 
				\item\label{it:i2} $|C_{i}| > 4n-10s$;
				\item\label{it:i3} $Y_i\cap C_{i} =X_{i+1}\cap C_i =\vn$.
			\end{enumerate}
	\end{claim}
	
	\begin{proof}
		Assume first that such a configuration exists for $r=2$.
		We construct a graph $G^+$ by adding six new 
		vertices $x_1$, $y_1$, $c_1$, $x_2$, $y_2$, $c_2$ to $G$ as well as 
		all edges from $x_i$ to $X_i$, from $y_i$ to $Y_i$, from $c_i$ to $C_i$ 
		(where $i=1,2$), and finally the hexagon $x_1y_1c_1x_2y_2c_2$ 
		(see Figure~\ref{fig:G+6}).
		
			\begin{figure}[ht]
	\centering
	\begin{tikzpicture}[scale=.8]
		
	\coordinate (a1) at (2.2, 1.9);
	\coordinate (a2) at (4.4, 1.6);
	\coordinate (a3) at (6.6, 1.5);
	\coordinate (a4) at (8.8, 1.5);
	\coordinate (a5) at (11, 1.6);
	\coordinate (a6) at (13.2, 1.9);
	
		\foreach \i in {1,...,6} {
					\coordinate (b\i) at (2.2*\i, 0);	
		}
	
		\foreach \i in {1,4} {
		\fill [red!10] (a\i)--($(b\i)+(.97,.16)$) -- ($(b\i)-(.97,-.16)$) -- cycle;
			\fill (a\i) circle (2pt);		
		\draw [red!70!black] (b\i) ellipse (1cm and .5cm);	
		\fill [red!10] (b\i) ellipse (1cm and .5cm);	
	}
	
	\foreach \i in {2, 5} {
		\fill [blue!10] (a\i)--($(b\i)+(.97,.16)$) -- ($(b\i)-(.97,-.16)$) -- cycle;
			\fill (a\i) circle (2pt);		
		\draw [blue!70!black] (b\i) ellipse (1cm and .5cm);
		\fill [blue!10] (b\i) ellipse (1cm and .5cm);	
	}

	\foreach \i in {3,6} {
	\fill [green!10!white] (a\i)--($(b\i)+(.62,.05)$) -- ($(b\i)-(.62,-.05)$) -- cycle;
		\fill (a\i) circle (2pt);		
	\draw [green!70!black,thick] (b\i) ellipse (.6cm and .3cm);	
	\fill [green!10!white] (b\i) ellipse (.6cm and .3cm);
}
	
		\draw (a1)--(a2)--(a3)--(a4)--(a5)--(a6)--(a1);
		
	\node at (b2) {$Y_1$};
	\node at (b3) {\scriptsize $C_1$};
	\node at (b6) {\scriptsize $C_2$};
	\node at (b1) {$X_1$};
	\node at (b4) {$X_2$};
	\node at (b5) {$Y_2$};
		
		\node at ($(a6)+(.35,.01)$) {\tiny $c_2$};
		\node  at ($(a1)+(-.27,.01)$) {\tiny $x_1$};

		\node at ($(a4)+(.34,-.14)$) {\tiny $x_2$};
		\node  at ($(a5)+(.34,-.14)$) {\tiny $y_2$};
		\node at ($(a2)+(-.3,-.14)$) {\tiny $y_1$};
		\node  at ($(a3)+(-.3,-.14)$) {\tiny $c_1$};

	\end{tikzpicture}
	\caption{The graph $G^+$. In reality, some of  
	$X_1$, $Y_1$, $C_1$, $X_2$, $Y_2$, $C_2$ overlap.}
	\label{fig:G+6}
\end{figure}
 			
		Due to~\ref{it:i1} and~\ref{it:i3} this graph is triangle-free. 
		Moreover, $\delta(n+6, s+2)=\delta-2$ and the minimality of $(n, s)$
		reveal 
				\begin{align*}
				\ex(n+6, s+2) &\le g_4(n+6, s+2) = g_4(n,s) + 4s + 2(4n-10s)+8\\
				&< \ex(n,s) + |X_1|+|Y_1|+|C_1|+|X_2|+|Y_2|+|C_2| + 6=e(G^+)\,,
		\end{align*}
				for which reason $G^+$ contains an independent set $Z^{+}$ of size $s+3$. 
		Clearly, $Z^+$ contains at most~$s$ old vertices and at most three new
		ones. Thus there is some $Z\in V(\cF)$ such that~$Z^+$ is either 
		$Z\cup\{x_1, y_2, c_1\}$ or $Z\cup\{x_2, y_1, c_2\}$. 
		In both cases $Z$ has a neighbour in $\cX$ and a neighbour in $\cY$, 
		which contradicts the fact that $\cX\dcup \cY$ bipartises $\cF$.
		This proves the case $r=2$ of our assertion. 
		
		Now we keep assuming that our claim fails and consider a counterexample 
		with $r$ minimum. As we have just seen, $r$ is at least $3$. 
		By Lemma~\ref{cl:1} applied to $X_1Y_1$ and $C_2$ there exists 
		a set $Z\in V(\cF)$ such that $C_2\cap Z = \vn$ and either $X_1Z$ 
		or $Y_1Z$ is an edge of~$\cF$.  
		
		Suppose first that $X_1Z\in E(\cF)$, which yields $Z\in \cY$. 
		Now the sets $X_1, X_3, \dots, X_r\in \cX$, $Z, Y_3, \dots, Y_r\in \cY$,
		and $C_2, \dots, C_r\subseteq V(G)$ contradict the minimality of $r$. 
		This proves $Y_1Z\in E(\cF)$, whence $Z\in \cX$. But now the sets 
		$Z, X_2, \in \cX$, $Y_1, Y_2\in\cY$, and $C_1, C_2\subseteq V(G)$
		yield the same contradiction.  
	\end{proof}
	
	\begin{claim}\label{cl:4}
		There do not exist $r\ge 1$, $X_1, \dots, X_r\in \cX$, and $Y_1, \dots, Y_r\in \cY$ with 
		\begin{enumerate}[label=\rmlabel]
			\item\label{it:ii1} $X_iY_i\in E(\cF)$,
			\item\label{it:ii2} $|X_i\cap Y_{i+1}|\ge 3s-n$ 
		\end{enumerate}
	for all $i\in \ZZ/r\ZZ$.
	\end{claim}
	
	\begin{proof}
		Assume contrariwise that such a situation exists. Since $X_1$ 
		is adjacent to $Y_1$ but not to $Y_2$, we have $r\ge 2$. 
		For every $i\in \ZZ/r\ZZ$ Lemma~\ref{lem:q} and~\ref{it:ii2}  
		yield a vertex $c_i\in X_i\cap Y_{i+1}$ such that the 
		cardinality of $C_i=\Nn(c_i)$ exceeds $4n-10s$. As the sets~$C_i$
		are disjoint to $X_i$ and $Y_{i+1}$, they lead to a contradiction 
		to Claim~\ref{cl:2}. 
	\end{proof}
		
	Let us now consider the sets 
		\begin{align*}
		\cX' &= \{X\in \cX\colon |X\cap Y| < 3s-n \text{ for every } Y\in\cY\} \\
		\text{and} \quad
		\cY' &= \{Y\in \cY \colon |X\cap Y| < 3s-n \text{ for every } X\in\cX\}\,.
	\end{align*}
			
	\begin{claim}\label{cl:5}
		There exists a vertex in $\cY$ all of whose neighbours are 
		in $\cX'$.
	\end{claim}
		\begin{proof}
		Recall that $\cF$ has no isolated vertices by Lemma~\ref{cl:62}. Thus the failure of our 
		claim would imply that every $Y\in\cY$ has a neighbour 
		in $\cX\setminus\cX'$. Starting with an arbitrary vertex $Y_1\in \cY$
		this allows us to construct recursively an infinite sequence  
	   $Y_1, X_1, Y_2, X_2, \dots$ such that for every~$i\in \NN$ 
	   we have
				\begin{enumerate}[label=\rmlabel]
			\item $Y_i \in \cY$, $X_i\in \cX\setminus \cX'$;
			\item $X_iY_i\in E(\cF)$;
			\item $|X_i\cap Y_{i+1}|\ge 3s-n$.
		\end{enumerate}
				
		Since $\cY$ is finite, there need to exist indices $p<q$ such that 
		$Y_p=Y_q$. But now the cyclic sequences $X_p, \dots, X_{q-1}\in\cX$ 
		and $Y_p, \dots, Y_{q-1}\in\cY$ of length $r=q-p$ contradict the 
		previous claim.
	\end{proof}
	
	As $\cF$ has no isolated vertices, Claim~\ref{cl:5} yields, in particular, 
	$\cX'\ne\vn$ and by symmetry~$\cY'$ cannot be empty either. 
	Due to the definitions of~$\cX'$ and~$\cY'$, the proof of our lemma can be 
	completed by showing that there is an edge $XY$ with $X\in \cX'$
	and $Y\in \cY'$.
	
	To this end we pick a vertex $Y_\star\in\cY$ such 
	that $\Nn_\cF(Y_\star)\subseteq \cX'$, a vertex $Y\in \cY'$, as well as 
	an arbitrary neighbour $X_\star$ of $Y$ (see Figure~\ref{fig:G+}). 
		\begin{figure}[ht]
	\centering
	\begin{tikzpicture}[scale=0.55]

		\coordinate (a1) at (-3, 2.5);
		\coordinate (a2) at (1,2.5);
		\coordinate (b1) at (-3,-2.5);
		\coordinate (b2) at (1, -2.5);

				\draw[red!75!black, line width=2pt] (-.6,2.5) ellipse (6cm and 1.2cm);
				\fill [red!02, thick] (-.6,2.5) ellipse (6cm and 1.2cm);
				
					\draw[blue!75!black, line width=2pt] (-.6,-2.5) ellipse (6cm and 1.2cm);
					\fill [blue!02, thick] (-.6,-2.5) ellipse (6cm and 1.2cm);

			\draw[red!75!black, line width=2pt] (1,2.5) ellipse (3cm and .8cm);
			\fill[red!07] (1,2.5) ellipse (3cm and .8cm);

			\draw[blue!75!black, line width=2pt] (1,-2.5) ellipse (3cm and .8cm);
			\fill[blue!07] (1,-2.5) ellipse (3cm and .8cm);

			\fill [black!15] (b1) -- (0,2.5) -- (2,2.42) -- cycle;
			
			\draw[thick] (1,2.5) ellipse (1cm and .3cm);		
			\fill[black!20] (1,2.5) ellipse (1cm and .3cm);

		\draw[dashed] (a2)--(b1);
		\draw[dashed] (a2)--(b2);
		\draw (a1)--(b2);
		\foreach \i in {a1,a2,b1,b2} \fill (\i) circle (3pt);
		
	\node at (2.8,2.5) {\large\textcolor{red!75!black}{$\cX'$}};
	\node at (2.8,-2.5) {\large\textcolor{blue!75!black}{$\cY'$}};
	\node at (-5.6,2.5) {\large\textcolor{red!75!black}{$\cX$}};
	\node at (-5.6,-2.5) {\large\textcolor{blue!75!black}{$\cY$}};
	
	\node at (-.8,2.6) {\tiny $\Nn(Y_\star)$};
	\node [above] at (a1) {\tiny $X_\star$};
	\node [left] at (a2) {\tiny $X$};
	\node [below] at (b1) {\tiny $Y_\star$};
	\node [right] at (b2) {\tiny $Y$};
			
	\end{tikzpicture}
	\caption{Obtaining the edge $XY$.}
	\label{fig:G+}
\end{figure}
 	By Corollary~\ref{cor:n} applied to the edge $X_\star Y$ and the 
	vertex $Y_\star$ there is a neighbour $X$ of~$Y_\star$ adjacent 
	to either $X_\star$ or $Y$.
	Our choice of~$Y_\star$ guarantees $X\in\cX'$ and the independence 
	of~$\cX$ yields $XX_\star\not\in E(\cF)$. Thus $XY$ is the desired edge.  
\end{proof}

Now all that still separates us from the main result are one computation 
and two symmetrisations. 
 
\begin{fact}\label{f:2redux}
	If $(n, s)$ is a minimal counterexample, then 
		\begin{equation}\label{eq:43}
			\ex(n-2,s-1) < \ex(n,s) - (2s-1)\,.
	\end{equation}
	\end{fact}

\begin{proof}
	For $\delta = 1$ the trivial bound~\eqref{eq:trivial} implies
	\begin{align*}
		2\ex(n-2, s-1) 
		&\le 
		(n-2)(s-1)=2g_4(n,s)+(11s-4n)(3n-8s) -n -2s +2\\
		&<
		2\ex(n,s) +2(n-3s) - 2(2s-1) 
		<2\ex(n,s) - 2(2s-1)\,.
	\end{align*}
	Otherwise $\delta \ge 2$, and $\delta(n-2, s-1)=\delta-3$ combined with
	the minimality of $(n, s)$ yields
	\begin{align*}
		\ex(n-2,s-1) 
		&\le 
		g_4(n-2,s-1) = g_4(n,s) + 8n-24s+4 \\
		&< 
		\ex(n,s) -2\delta - 2s + 4
		 < \ex(n,s) - (2s-1)\,. \qedhere
	\end{align*}
\end{proof}

\begin{proof}[Proof of Theorem~\ref{th:main}]
	If the result failed, there would exist a minimal counterexample~$(n, s)$. 
	Let $G\in\Ex$ denote an arbitrary extremal graph. 
	According to Lemma~\ref{lem:56} there exists an edge $A_1A_2\in E(\cF)$ such 
	that every $Z\in V(\cF)$ satisfies $|A_1\cap Z|<3s-n$ or $|A_2\cap Z|<3s-n$.
	
	Now two successive applications of Corollary~\ref{lem:B'} lead to sets  
	$B_1\subseteq A_2$ and $B_2\subseteq A_1$ of size $3s-n$ such that 
	the graph $G_\star$ obtained from $G$ by symmetrising first 
	$\GZS(A_1, B_1)$ and then $\GZS(A_2,B_2)$ still belongs to $\Ex$.
	
	Pick arbitrary vertices $b_i\in B_i$ and set $G^-= G_\star- \{b_1, b_2\}$. 
	Due to 
		\[
		e(G^-) = e(G) - (2s-1) = \ex(n,s) - (2s-1) 
		\overset{\eqref{eq:43}}{>} 
		\ex(n-2,s-1)
	\]
		there is an independent set $Z$ of size $s$ in $G^-$.
	But now for $i=1, 2$ the set $Z\cup \{b_{i}\}$ cannot be independent 
	in $G_\star$, which proves $Z\cap A_{i}\neq \vn$ and 
	thus $Z\cap B_i=\vn$. 		
	Consequently,~$Z$ was already independent in $G$. Moreover, the  
	sets $(Z\setminus A_i)\cup B_i$ are independent in~$G_\star$, 
	whence $|Z\cap A_i|\ge |B_i|=3s-n$. Altogether $Z$ contradicts our 
	choice of the edge~${A_1A_2\in E(\cF)}$.   	 	
\end{proof}

\subsection*{Acknowledgement} 
We would like to thank the referees for reading our work very carefully. 

\subsection*{Added in proof}
While this article was in press, we completed the next step towards proving 
Andra\'sfai's conjecture~\cite{LPR-TRT1}.

\end{document}